\documentclass[11pt]{amsart}

\usepackage{amsmath,amssymb,amsthm}

\usepackage[pdftex]{graphicx}
\usepackage{subcaption}
\usepackage{float}
\usepackage{color}
\numberwithin{equation}{section}

\newtheorem{theorem}{Theorem}
\newtheorem{proposition}[theorem]{Proposition}

\newtheorem{lemma}[theorem]{Lemma}

\newcommand{\reals}{\mathbf{R}}

\title[Profile Decomposition for Hyperbolic Schr\"odinger]{The profile decomposition for the hyperbolic Schr\"odinger equation}

\author[B. Dodson]
{Benjamin Dodson}
\email{dodson@math.jhu.edu}
\address{Mathematics Department, Johns Hopkins University  \\
Baltimore, MD, USA}

\author[J.L. Marzuola]
{Jeremy L. Marzuola}
\email{marzuola@math.unc.edu}
\address{Mathematics Department, University of North Carolina \\
 Chapel Hill, NC 27599, USA}

\author[B. Pausader]
{Benoit Pausader}
\email{benoit.pausader@math.brown.edu}
\address{Mathematics Department, Brown University \\
Providence, RI, USA}

\author[D. Spirn]
{Daniel P. Spirn}
\email{spirn@math.umn.edu}
\address{Mathematics Department, University of Minnesota \\
Minneapolis, MN, USA}

\begin{document}


\thanks{}

\begin{abstract}
In this note, we prove the profile decomposition for hyperbolic Schr{\"o}dinger (or mixed signature) equations on $\reals^2$ in two cases, one mass-supercritical and one mass-critical.  First, as a warm up,  we show that the profile decomposition works for the ${\dot H}^{\frac12}$ critical problem. Then, we give the derivation of the profile decomposition in the mass-critical case  based on an estimate of Rogers-Vargas \cite{RV}.
\end{abstract}

\maketitle

\section{Introduction}

We will consider the hyperbolic (or mixed signature) Schr\"odinger equation on $\mathbf{R}^2$, which is given by
\begin{equation}\label{HNLS}
i \partial_t  u + \partial_{x} \partial_{y} u = |u|^{p} u, \hspace{5mm} u(x, y,0) = u_{0}(x, y).
\end{equation}
In particular, we will focus on the cases $p=4$ and $p=2$.  The case $p=2$
arises naturally in the study of modulation of wave trains in gravity water waves, see for instance \cite{T2,TW}; it is also a natural component of the Davey-Stewartson system \cite{LP,SS}.  As can be observed quickly from the nature of the dispersion relation, the linear problem
\begin{equation}\label{HLS}
i \partial_tu + \partial_{x} \partial_{y} u = 0, \hspace{5mm} u(x,y,0) = u_{0}(x, y).
\end{equation}
satisfies the same Strichartz estimates and rather similar local smoothing estimates\footnote{See \cite{Chihara,RS} for a general treatment of smoothing estimates for dispersive equations with non-elliptic symbols.} as its elliptic counterpart, the standard Schr\"odinger equation. In particular, there exists a constant $C<\infty$ such that
\begin{equation}\label{LinStric}
\Vert e^{it\partial_{x}\partial_{y}}f\Vert_{L^4_{x,y,t}}\le C\Vert f\Vert_{L^2_{x,y}}.
\end{equation}
Hence, large data local in time well-posedness and global existence for small data with $p\geq2$ can be observed using standard methods that can be found in classical texts such as \cite{Cazenave,SS}.  For quasilinear problems with mixed signature, some local well-posedness results have been developed recently, see \cite{KPRV1,MMT3}.  Non-existence of bound states was established in \cite{GS} and a class of bound states that are not in $L^2$ were constructed in \cite{KNZ}.

Long time low regularity theory for this equation at large data remains unknown however.  Recently, an approach to global existence for sufficiently regular solutions was taken in \cite{T1}, but it is conjectured  that \eqref{HNLS} should have global well-posedness and scattering for all initial data in $L^2$. Much progress has been made recently in proving global well-posedness and scattering for various critical and supercritical dispersive equations by applying concentration compactness tools, which originated with the works of Lions \cite{Lions1,Lions2}.  One major step in applying modern concentration compactness tools to dispersive equations is the profile decomposition, see \cite{KM1,KV}.  The idea is that given a small data global existence result, one proves that if the large data result is false then there is a critical value of norm of the initial data at which for instance a required integral fails to be finite.  Then, the profile decomposition ensures that failure occurs because of an almost periodic critical element, which may then be analyzed further and in ideal settings ruled out completely.  See \cite{Dod1,Dod2} and references therein for applications of this idea in the setting of focusing and defocusing Schr\"odinger equations for instance.

A major breakthrough in profile decompositions arose in the works of G\'erard \cite{G1}, Merle-Vega \cite{MerleVega}, Bahouri-G\'erard \cite{BG}, Gallagher \cite{Gal1} and Keraani \cite{Ker1}.  Those results have then been used to understand how to prove results about scattering, blow-up and global well-posedness in many settings, see \cite{KV} for some examples. We also mention the recent work by Fanelli-Visciglia \cite{FV}, where they consider profile decompositions in mass super-critical problems for a variety of operators, including \eqref{HNLS}.

As can be seen in \cite[Section 4.4]{KV}, the profile decomposition follows from refined bilinear Strichartz estimates. Using refined Strichartz estimates from \cite{MVV} and bilinear Strichartz estimates, Bourgain \cite{Bou1} proved concentration estimates and global well-posedness in $H^{3/5 + \epsilon}$ for the defocusing, cubic elliptic nonlinear Schr{\"o}dinger equation in $\reals^2$. Building on this work, Merle-Vega \cite{MerleVega} proved a profile decomposition for the mass-critical elliptic nonlinear Schr{\"o}dinger equation in two dimensions.

For the hyperbolic NLS, Lee, Vargas and Rogers-Vargas \cite{Lee,RV,V1} have provided refined linear and bilinear estimates, drawing  on results of Tao \cite{Tao1} for the elliptic Schr\"odinger equation. In particular \cite{RV} gives an improved Strichartz estimate similar to our Proposition \ref{PD1} and uses it to prove lower bounds on concentration of mass at blow-up. An improved Strichartz estimate is also the key element in our profile decomposition, following the standard machinery described in \cite[Section 4.4]{KV}. For completeness, we provide a proof of Proposition \ref{PD1}, which, although drawing on similar ideas as in \cite{RV}, outlines more explicitly the additional orthogonality of rectangles with skewed ratios.

The major issue in following the standard proof of the profile decomposition is that the mixed signature nature of \eqref{HNLS} means that an essential bilinear interaction estimate that holds in the elliptic case fails.  This is compensated for in \cite{V1} by making a required orthogonality assumption for the refined bilinear Strichartz to hold (see the statement in Lemma \ref{BS2} below). To overcome this difficulty, we use a double Whitney decomposition to precisely identify the right scales, which introduces many different rectangles that are controlled using the fact that functions with support on two rectangles of different aspect ratios have small bilinear interactions.  We note that while we here focus on analysis in $2$ dimensions to keep the technical computations focused and directed, we expect many of the calculations to be generalizable to other dimensions as in \cite{KV}.

The paper is structured as follows: in Section \ref{props}, we set up the problem, discuss some basic symmetries and establish some important bilinear estimates; in Section \ref{super} we establish the result in the mass-supercritical case using the extra compactness that comes from the Sobolev embedding; in Section \ref{crit}, we establish the main precise Strichartz estimate in the paper and in Section \ref{SecMBUS}, we obtain the profile decomposition for the mass-critical problem and deduce the existence of a minimal blow-up solution. Finally, we conclude with a short Appendix in which we correct an error in the published version of this article.

{\sc Acknowledgments.} {The first author was supported in part by
  U.S. NSF Grant DMS - 1500424  The second author was supported in
  part by U.S. NSF Grants DMS--1312874 and DMS-1352353. The third author was supported in part by
  U.S. NSF Grant DMS--1558729 and a Sloane Research fellowship.  The fourth author was supported in part by
  U.S. NSF Grant DMS-1516565.  We wish to
  thank Andrea Nahmod, Klaus Widmayer, Daniel Tataru, Nathan Totz for helpful conversations during the
  production of this work. Part of this work was initiated when some of the authors were at the Hausdorff Research Institute for Mathematics in Bonn, then progressed during visits to the Institut des Hautes \'Etudes in Paris, the Mathematical Sciences Research Institute in Berkeley and the Institute for Mathematics and Applications in Minneapolis.  The authors would like to thank these institutions for hosting subsets of them at various times in the last several years.  
  
We thank also an anonymous referee who pointed out the results of Rogers-Vargas \cite{RV} and its relevance for this work and Emanuel Carneiro for pointing out that error in the original appendix and sharing his manuscript on optimizers for \eqref{LinStric}, \cite{COS}.
  }

\section{Properties of \eqref{HNLS}}
\label{props}

\noindent Observe that a solution to 

\begin{equation}\label{1.1}
i \partial_tu + \partial_{x} \partial_{y} u = |u|^{2} u, \hspace{5mm} u(x, y,t) = u_{0}(x, y),
\end{equation}

\noindent has a number of symmetries:

\begin{enumerate}
\item Translation: for any $(x_{0}, y_{0}) \in \mathbf{R}^{2}$,
\begin{equation}\label{1.2}
u(x,y,t) \mapsto u(x - x_{0}, y - y_{0},t),
\end{equation}

\item Modulation: for any $\theta \in \mathbf{R}$,
\begin{equation}\label{1.5}
u(x,y,t) \mapsto e^{i \theta} u(x,y,t).
\end{equation}

\item Scaling: for any $\lambda_{1}, \lambda_{2} > 0$,
\begin{equation}\label{1.3}
u(x,y,t) \mapsto \sqrt{\lambda_{1} \lambda_{2}} u(\lambda_{1} x, \lambda_{2} y,\lambda_{1} \lambda_{2} t),
\end{equation} 

\item Galilean symmetry: for $(\xi_{1}, \xi_{2}) \in \mathbf{R}^{2}$,
\begin{equation}\label{1.4}
u(x,y,t) \mapsto e^{-it \xi_{1} \xi_{2}} e^{i[x \xi_{1} + y \xi_{2}]} u(x - \xi_{1} t, y - \xi_{2} t,t).
\end{equation}

\item Pseudo-conformal symmetry:
\begin{equation}\label{1.5_1}
u(x,y,t) \mapsto \frac{e^{i\frac{xy}{t}}}{it}\overline{u}(\frac{x}{t},\frac{y}{t},\frac{1}{t}).
\end{equation}

\end{enumerate}

\noindent These symmetries all preserve the $L^{2}_{x,y}$ norm. The first two symmetries $\eqref{1.2}$-$\eqref{1.5}$, as well as the scaling symmetry properly redefined, also preserve the $\dot{H}_{h}^{s}$ norm for any $s \in \mathbf{R}$, where

\begin{equation*}
\begin{split}
\Vert f\Vert_{\dot{H}_h^s}^2&=\Vert \vert\partial_x\vert^\frac{s}{2}\vert \partial_y\vert^\frac{s}{2}f\Vert_{L^2}^2.
\end{split}
\end{equation*}

\noindent Note that this norm has similar scaling laws as the more usual $\dot{H}^{s}$ norm.  Other examples of anisotropic equations have appeared in for instance \cite{S-T,Iftimie}. For example, for the $\dot{H}^{1/2}$ - critical problem

\begin{equation}\label{1.6}
i \partial_tu + \partial_{x} \partial_{y} u = |u|^{4} u, \hspace{5mm} u(x,y,0) = u_{0}(x, y),
\end{equation}

\noindent the symmetries are thus:

\begin{enumerate}
\item Translation: for any $(x_{0}, y_{0}) \in \mathbf{R}^{2}$,

\begin{equation}\label{1.7}
u(x,y,t) \mapsto u(x - x_{0}, y - y_{0},t),
\end{equation}

\item Scaling: for any $\lambda_{1}, \lambda_{2} > 0$,

\begin{equation}\label{1.8}
u(x,y,t) \mapsto (\lambda_{1} \lambda_{2})^{1/4} u(\lambda_{1} x, \lambda_{2} y,\lambda_{1} \lambda_{2} t),
\end{equation}

\item Modulation: for any $\theta \in \mathbf{R}$,

\begin{equation}\label{1.9}
u(x,y,t) \mapsto e^{i \theta} u(x,y,t).
\end{equation}
\end{enumerate}

\noindent We will treat the profile decomposition for $\eqref{1.6}$ as a warm-up, before tackling the profile decomposition for the mass-critical problem $\eqref{1.1}$.

\subsection{Notations}

Let $\varphi$ be a usual smooth bump function such that $\varphi(x)=1$ when $\vert x\vert\le 1$ and $\varphi(x)=0$ when $\vert x\vert\ge 2$. We also let
\begin{equation*}
\psi(x)=\varphi(x)-\varphi(2x).
\end{equation*}

We will often consider various projections in Fourier space. Given a rectangle $R=R(c,\ell_x,\ell_y)$, centered at $c=(c_x,c_y)$ and with sides parallel to the axis of length $2\ell_x$ and $2\ell_y$, we define
\begin{equation}
\begin{split}
\varphi_{R}(x,y)=\varphi(\ell_x^{-1}(x-c_x))\varphi(\ell_y^{-1}(y-c_y)).
\end{split}
\label{phiRdef}
\end{equation}
We define the operators
\begin{equation*}
\begin{split}
\widehat{Q_{M,N}f}(\xi,\eta)&=\psi(M^{-1}\xi)\psi(N^{-1}\eta)\widehat{f}(\xi,\eta),\\
\widehat{P_Rf}(\xi,\eta)&=\varphi_{R}(\xi,\eta)\widehat{f}(\xi,\eta).
\end{split}
\end{equation*}
The first operator is only sensitive to the scales involved, while the second also accounts for the location in Fourier space.
We also let $\vert R\vert=4\ell_x\ell_y$ denote its area.

\subsection{Some Preliminary Estimates}

We start with a nonisotropic version of the Sobolev embedding.
\begin{lemma}\label{Sob}

There holds that
\begin{equation*}
\Vert f\Vert_{L^q_{x,y}}\lesssim \Vert \vert\partial_x\vert^\frac{s}{2}\vert\partial_y\vert^\frac{s}{2}f\Vert_{L^p_{x,y}}
\end{equation*}
whenever $1<p\le q<\infty$, $0\le s<1$ and
\begin{equation*}
\frac{1}{q}=\frac{1}{p}-\frac{s}{2}
\end{equation*}
\end{lemma}

\begin{proof}[Proof of Lemma \ref{Sob}]

The proof, although easy, highlights the need to treat each direction independently. Using Sobolev embedding in $1d$, Minkowski inequality and Sobolev again, we obtain that
\begin{equation*}
\begin{split}
\Vert f\Vert_{L^{q}_x(\mathbf{R},L^{q}_y(\mathbf{R}))}&\lesssim \Vert \vert\partial_x\vert^\frac{s}{2}f\Vert_{L^{q}_y(\mathbf{R},L^p_x(\mathbf{R}))}\lesssim \Vert \vert\partial_x\vert^\frac{s}{2}f\Vert_{L^{p}_x(\mathbf{R},L^q_y(\mathbf{R}))}  \\
& \lesssim \Vert \vert\partial_y\vert^\frac{s}{2}\vert\partial_x\vert^\frac{s}{2}f\Vert_{L^p_{x,y}}
\end{split}
\end{equation*}
which is what we wanted.
\end{proof}

We have two basic refinements of \eqref{LinStric}. Note the difference in orthogonality requirements between Lemma \ref{BS1} and Lemma \ref{BS2}.
\begin{lemma}\label{BS1}
Assume that $f=P_{R_1}f$ and $g=P_{R_2}g$ where $R_i=R(c^i,\ell_x,\ell_y)$ and $\vert c^1_x-c^2_x\vert = N\ge 4\ell_x$, and let $u$ (resp. $v$) be a solution of \eqref{HLS} with initial data $f$ (resp. $g$). Then
\begin{equation}\label{BSRes1}
\Vert uv\Vert_{L^2_{x,y,t}}\lesssim \left(\frac{\ell_x}{N}\right)^\frac{1}{2}\Vert f\Vert_{L^2_{x,y}}\Vert g\Vert_{L^2_{x,y}}.
\end{equation}

\end{lemma}

\begin{lemma}\label{BS2}

Assume that $f=P_{R_1}f$ and $g=P_{R_2}g$ where $R_i=R(c^i,\ell_x,\ell_y)$, $\vert c^1_x-c^2_x\vert\ge 4  \ell_x$ and $\vert c^1_y-c^2_y\vert\ge 4\ell_y$, and let $u$ (resp. $v$) be a solution of \eqref{HLS} with initial data $f$ (resp. $g$). Then
\begin{equation*}
\Vert uv\Vert_{L^q_{x,y,t}}\lesssim (\ell_x\ell_y)^{1-\frac{2}{q}}\Vert f\Vert_{L^2_{x,y}}\Vert g\Vert_{L^2_{x,y}}
\end{equation*}
whenever $q>5/3$.
\end{lemma}

Lemma \ref{BS2} is the main refined bilinear estimate and appears essentially in \cite{V1} when dealing with cubes.   The result as stated here follows by scaling rectangles to cubes.

\begin{proof}[Proof of Lemma \ref{BS1}] We simply write that
\begin{equation*}
\begin{split}
\widehat{u^2}(\xi,\eta,t)&=I(\xi,\eta,t),\\
I(\xi,\eta,t)&=\iint_{\mathbf{R}}e^{-it\omega}\varphi_{R_1}(\xi_1,\eta_1)\varphi_{R_2}(\xi-\xi_1,\eta-\eta_1)  \\
& \hspace{2cm} \times \widehat{f}(\xi_1,\eta_1)\widehat{g}(\xi-\xi_1,\eta-\eta_1)d\xi_1 d\eta_1,\\
\omega&=\xi_1\eta_1+(\xi-\xi_1)(\eta-\eta_1)
\end{split}
\end{equation*}
we may now change variable in the integral
\begin{equation}\label{ChangeofVar}
\begin{split}
(\xi_1,\eta_1)\mapsto (\xi_1,\omega),\qquad J:=\frac{\partial(\xi_1,\omega)}{\partial(\xi_1,\eta_1)}=\begin{pmatrix}1&0\\2\eta_1-\eta&2\xi_1-\xi\end{pmatrix}
\end{split}
\end{equation}
and in particular, we remark that
\begin{equation}\label{JSize}
\vert J\vert=\vert (\xi-\xi_1)-\xi_1\vert \simeq N,
\end{equation}
so that
\begin{equation*}
\begin{split}
I(\xi,\eta,t) = & \iint_{\mathbf{R}}e^{-it\omega}\varphi_{R_1}(\xi_1,\eta_1)\varphi_{R_2}(\xi-\xi_1,\eta-\eta_1) \\
& \hspace{1cm} \times \widehat{f}(\xi_1,\eta_1)\widehat{g}(\xi-\xi_1,\eta-\eta_1)\cdot J^{-1} d\xi_1 d\omega,\\
\eta_1&=\eta_1(\xi_1,\omega;\xi,\eta). \\
\end{split}
\end{equation*}
Taking into consideration the Fourier transform in time and using Plancherel, followed by Cauchy-Schwarz, we find that
\begin{equation*}
\begin{split}
& \Vert I(\xi,\eta,\cdot)\Vert_{L^2_t}^2=  \\
& \int_{\mathbf{R}}\left\vert\int_{\mathbf{R}}\varphi_{R_1}(\xi_1,\eta_1)\varphi_{R_2}(\xi-\xi_1,\eta-\eta_1)\widehat{f}(\xi_1,\eta_1)\widehat{g}(\xi-\xi_1,\eta-\eta_1)\cdot J^{-1} d\xi_1\right\vert^2 d\omega\\
&\le\sup_{\xi,\eta,\eta_1}\int_{\mathbf{R}}\varphi_{R_1}(\xi_1,\eta_1)\varphi_{R_2}(\xi-\xi_1,\eta-\eta_1)d\xi_1\times\\
&\quad\int_{\mathbf{R}}\int_{\mathbf{R}}\varphi_{R_1}(\xi_1,\eta_1)\varphi_{R_2}(\xi-\xi_1,\eta-\eta_1)\left\vert\widehat{f}(\xi_1,\eta_1)\widehat{g}(\xi-\xi_1,\eta-\eta_1)\right\vert^2\cdot J^{-2} d\xi_1 d\omega.\\
\end{split}
\end{equation*}
Now, we use the fact that $R_1$ has width $\ell_x$, together with \eqref{JSize} to obtain, after undoing the change of variables, that
\begin{equation*}
\begin{split}
\Vert I(\xi,\eta,\cdot)\Vert_{L^2_t}^2&\lesssim\frac{\ell_x}{N}\int_{\mathbf{R}}\int_{\mathbf{R}}\varphi_{R_1}(\xi_1,\eta_1)\varphi_{R_2}(\xi-\xi_1,\eta-\eta_1) \\ 
& \hspace{2cm} \left\vert\widehat{f}(\xi_1,\eta_1)\widehat{g}(\xi-\xi_1,\eta-\eta_1)\right\vert^2\cdot J^{-1} d\xi_1 d\omega\\
&\lesssim \frac{\ell_x}{N}\int_{\mathbf{R}}\int_{\mathbf{R}} \left\vert\widehat{f}(\xi_1,\eta_1)\widehat{g}(\xi-\xi_1,\eta-\eta_1)\right\vert^2 d\xi_1 d\eta_1.\\
\end{split}
\end{equation*}
Integrating with respect to $(\xi,\eta)$, we then obtain \eqref{BSRes1}.
\end{proof}


We will in fact use the following consequence of Lemma \ref{BS2}.
\begin{lemma}\label{LemBSE3}

Under the assumptions of Lemma \ref{BS2}, it holds that
\begin{equation}\label{BSE3}
\Vert uv  \Vert_{L^\frac{40}{21}_{x,y,t}}\lesssim (\ell_x\ell_y)^{-\frac{3}{20}}\Vert\widehat{f}\Vert_{L^\frac{20}{11}_{x,y}}\Vert \widehat{g}\Vert_{L^\frac{20}{11}_{x,y}}.
\end{equation}

\end{lemma}

\begin{proof} Indeed, using Lemma \ref{BS2}, we find that
\begin{equation*}
\Vert e^{it \partial_x \partial_y}f\cdot e^{it \partial_x \partial_y }g\Vert_{L^\frac{12}{7}_{x,y,t}}\lesssim (\ell_x\ell_y)^{-\frac{1}{6}}\Vert\widehat{f}\Vert_{L^2_{x,y}}\Vert \widehat{g}\Vert_{L^2_{x,y}},
\end{equation*}
while a crude estimate gives that
\begin{equation*}
\begin{split}
\Vert e^{it\partial_x \partial_y}f\cdot e^{it\partial_x \partial_y}g\Vert_{L^\infty_{x,y,t}}\lesssim \Vert\widehat{f}\Vert_{L^1_{x,y}}\Vert \widehat{g}\Vert_{L^1_{x,y}}.
\end{split}
\end{equation*}
Interpolation gives \eqref{BSE3}.
\end{proof}

Another tool we will need in the profile decomposition is the following local smoothing result which is essentially equivalent to Lemma \ref{BS1}:

\begin{lemma}\label{LocSmoothing}
Let $\phi\in L^2_{x,y}$. There holds that
\begin{equation*}
\begin{split}
\sup_x\Vert Q_{M,N}e^{it\partial_x\partial_y}\phi(x,\cdot)\Vert_{L^2_{y,t}}&\lesssim N^{-\frac{1}{2}}\Vert \phi\Vert_{L^2_{x,y}},\\
\sup_y\Vert Q_{M,N}e^{it\partial_x\partial_y}\phi(\cdot,y)\Vert_{L^2_{x,t}}&\lesssim M^{-\frac{1}{2}}\Vert \phi\Vert_{L^2_{x,y}}.\\
\end{split}
\end{equation*}
\end{lemma}

\begin{proof}
The proof is similar to the one in the elliptic case and follows from Plancherel after using a change of variable similar to \eqref{ChangeofVar}.  An equivalent statement with proof occurs in \cite[Theorem $2.1$]{LP}.  See also \cite{Chihara} for a general statement of Local Smoothing Estimates for Dispersive Equations.
\end{proof}

\section{Mass-supercritical HNLS}
\label{super}

In this section, we observe that $\dot{H}_{h}^{s}$ has similar improved Sobolev inequalities as the $\dot{H}^{1/2}$ Sobolev norm. A typical example is the following:

\begin{lemma}
\label{Soblem}
Let $f\in C^\infty_c(\mathbf{R}^2)$. There holds that
\begin{equation}\label{Sob1}
\begin{split}
\Vert f\Vert_{L^6_{x,y}}&\lesssim \left(\sup_{M,N}(MN)^{-\frac{1}{6}}\Vert Q_{M,N}f\Vert_{L^\infty}\right)^\frac{1}{3}\Vert f\Vert_{\dot{H}_h^\frac{2}{3}}^\frac{2}{3}\lesssim \Vert f\Vert_{\dot{H}_h^\frac{2}{3}},\\
\end{split}
\end{equation}
and consequently,
\begin{equation}\label{Sob2}
\Vert f\Vert_{L^4_{x,y}}\lesssim \left(\sup_{M,N}(MN)^{-\frac{1}{4}}\Vert Q_{M,N}f\Vert_{L^\infty}\right)^\frac{1}{6}\Vert f\Vert_{\dot{H}_h^\frac{1}{2}}^\frac{5}{6}\lesssim \Vert f\Vert_{\dot{H}_h^\frac{1}{2}}.
\end{equation}

\end{lemma}

This is essentially a consequence of the following simple inequalities
\begin{align}\label{Bern}
\begin{split}
\Vert Q_{M,N}f\Vert_{L^\infty_{x,y}}&\lesssim N^\frac{1}{2}\Vert Q_{M,N}f\Vert_{L^\infty_xL^2_y}\lesssim N^\frac{1}{2}\Vert Q_{M,N}f\Vert_{L^2_yL^\infty_x} \\
&  \lesssim (MN)^\frac{1}{2}\Vert Q_{M,N}f\Vert_{L^2_{x,y}},
\end{split}
\end{align}
and similarly after exchanging the role of $x$ and $y$.

\noindent \emph{Proof of Lemma \ref{Soblem}:}

Indeed, we may simply develop
\begin{equation*}
\begin{split}
\Vert f\Vert_{L^6_{x,y}}^6&\lesssim \sum_{\substack{M_1,\dots, M_6,\\ N_1,\dots, N_6}}\iint_{\mathbf{R}\times\mathbf{R}}Q_{M_1,N_1}f\cdot Q_{M_2,N_2}f\dots Q_{M_6,N_6}f dxdy
\end{split}
\end{equation*}
without loss of generality, we may assume that
\begin{equation*}
\begin{split}
M_5,M_6\lesssim\mu_2=\max_2\{M_1,M_2,M_3,M_4\},\\
N_5,N_6\lesssim\nu_2=\max_2\{N_1,N_2,N_3,N_4\},\\
\end{split}
\end{equation*}
where $\max_2(S)$ denotes the second largest element of the set $S$, and then using H\"older's inequality and summing over $M_5,M_6$ and $N_5,N_6$, we obtain
\begin{equation*}
\begin{split}
& \Vert f\Vert_{L^6_{x,y}}^6 \lesssim \left(\sup_{M,N}(MN)^{-\frac{1}{6}}\Vert Q_{M,N}f\Vert_{L^\infty}\right)^2\times\\
&\qquad \sum_{\substack{M_1,\dots, M_4,\,\, M_5,M_6\le\mu_2\\ N_1,\dots, N_4,\,\, N_5,N_6\le \nu_2}}(M_5M_6N_5N_6)^\frac{1}{6}\iint_{\mathbf{R}\times\mathbf{R}}\vert Q_{M_1,N_1}f\vert\dots\vert Q_{M_4,N_4}f\vert dxdy\\
&\lesssim \left(\sup_{M,N}(MN)^{-\frac{1}{6}}\Vert Q_{M,N}f\Vert_{L^\infty}\right)^2   \times \\
& \hspace{2cm} \sum_{\substack{M_1,\dots, M_4\\ N_1,\dots, N_4}}(\mu_2\nu_2)^\frac{1}{3} \iint_{\mathbf{R}\times\mathbf{R}}\vert Q_{M_1,N_1}f\vert\dots\vert Q_{M_4,N_4}f\vert dxdy.
\end{split}
\end{equation*}
Now, using \eqref{Bern} and estimating the norms corresponding to the two lower frequencies in each direction in $L^\infty$, and the two highest ones in $L^2$, one quickly finds that
\begin{equation*}
\begin{split}
\sum_{\substack{M_1,\dots, M_4\\ N_1,\dots, N_4}}(\mu_2\nu_2)^\frac{1}{3}\iint_{\mathbf{R}\times\mathbf{R}}\vert Q_{M_1,N_1}f\vert\dots\vert Q_{M_4,N_4}f\vert \, dxdy\lesssim \Vert f\Vert_{\dot{H}^\frac{2}{3}_h}^4,
\end{split}
\end{equation*}
which finishes the proof. Inequality \eqref{Sob2} then follows by interpolation.

\noindent $\Box$

\medskip

At this point, the usual profile decomposition follows easily from the following simple Lemma \ref{InverseStric1} below.
\begin{lemma}\label{InverseStric1}
There exists $\delta>0$ such that
\begin{equation*}
\Vert e^{it\partial_x \partial_y}f\Vert_{L^8_{x,y,t}}\lesssim \left(\sup_{M,N,t,x,y}(MN)^{-\frac{1}{4}}\vert \left(e^{it\partial_x \partial_y}Q_{M,N}f\right)(x,y)\vert\right)^\delta\Vert f\Vert_{\dot{H}^\frac{1}{2}_h}^{1-\delta}.
\end{equation*}
\end{lemma}

\begin{proof}[Proof of Lemma \ref{InverseStric1}]

We use H\"older's inequality, Sobolev embedding Lemma \ref{Sob}, Strichartz estimates and \eqref{Sob2} to get for $u=e^{it\partial_x\partial_y}f$
\begin{equation*}
\begin{split}
\Vert u\Vert_{L^8_{x,y,t}}&\lesssim \Vert u\Vert_{L^6_tL^{12}_{x,y}}^\frac{3}{4}\Vert u\Vert_{L^\infty_tL^4_{x,y}}^\frac{1}{4}\\
&\lesssim \Vert \vert\partial_x\vert^\frac{1}{4}\vert\partial_y\vert^\frac{1}{4}u\Vert_{L^6_tL^3_{x,y}}^\frac{3}{4}\cdot \left(\sup_{M,N,t}(MN)^{-\frac{1}{4}}\Vert Q_{M,N}u(t)\Vert_{L^\infty_{x,y}}\right)^\frac{1}{24}\Vert f\Vert_{\dot{H}^\frac{1}{2}_h}^\frac{5}{24}\\
&\lesssim \Vert f\Vert_{\dot{H}^\frac{1}{2}_h}^\frac{23}{24}\cdot \left(\sup_{M,N,t}(MN)^{-\frac{1}{4}}\Vert Q_{M,N}u(t)\Vert_{L^\infty_{x,y}}\right)^\frac{1}{24}.
\end{split}
\end{equation*}

\end{proof}

\subsection{The mass-supercritical profile decomposition}

Let us take the group action  on functions given by $g_{n}^{j} = g(x_{n}^{j}, y_{n}^{j}, \lambda_{1, n}^{j}, \lambda_{2, n}^{j})$ such that
\begin{align*}
& (g_{n}^{j})^{-1} f = (\lambda_{n, 1}^{j} \lambda_{n, 2}^{j})^{\frac14}  [f](\lambda_{n, 1}^{j} x + x_{n}^{j}, \lambda_{n, 2}^{j} y + y_{n}^{j}).
\end{align*}

\noindent We can now state the $\dot{H}^\frac{1}{2}_h$-profile decomposition for $(\ref{1.6})$.

\begin{proposition}
\label{supcritPD}
Let $\| u_{n} \|_{\dot{H}^\frac{1}{2}_h } \leq A$ be a sequence that is bounded $\dot{H}^\frac{1}{2}_h$. Then possibly after passing to a subsequence, for any $1 \leq j < \infty$ there exist $\phi^{j} \in \dot{H}^\frac{1}{2}_h$, $(t_{n}^{j}, x_{n}^{j}, y_{n}^{j}) \in \mathbf{R}^{3}$, $\lambda_{n, 1}^{j}$, $\lambda_{n, 2}^{j} \in (0, \infty)$ such that for any $J$,

\begin{equation}
u_{n} = \sum_{j = 1}^{J} g_{n}^{j} e^{i t_{n}^{j} \partial_{x} \partial_y} \phi^{j} + w_{n}^{J},
\end{equation}

\begin{equation}\label{strichartz_msc}
\lim_{J \rightarrow \infty} \limsup_{n \rightarrow \infty} \| e^{it \partial_x \partial_y  } w_{n}^{J} \|_{L_{x,y,t}^{8}} = 0,
\end{equation}

\noindent such that for any $1 \leq j \leq J$,

\begin{equation}
e^{-i t_{n}^{j} \partial_x \partial_y} (g_{n}^{j})^{-1} w_{n}^{J} \rightharpoonup 0,
\end{equation}

\noindent weakly in $\dot{H}^\frac{1}{2}_h$, 

\begin{equation}\label{decouple_msc}
\lim_{n \rightarrow \infty} \left(\| u_{n} \|_{\dot{H}^\frac{1}{2}_h}^{2} - \sum_{j = 1}^{J} \| \phi^{j} \|_{\dot{H}^\frac{1}{2}_h}^{2} - \| w_{n}^{J} \|_{\dot{H}^\frac{1}{2}_h}^{2}\right) = 0,
\end{equation}

\noindent and for any $j \neq k$,

\begin{align*}
& \limsup_{n\to\infty}\left[  \left\vert\ln \left( \frac{\lambda_{n, 1}^{j}}{\lambda_{n, 1}^{k}} \right)\right\vert + \left\vert \ln \left( \frac{\lambda_{n, 2}^{j}}{\lambda_{n, 2}^{k}}  \right)\right\vert   + \frac{|x_{n}^{j} - x_{n}^{k} |}{(\lambda_{n, 1}^{j} \lambda_{n, 1}^{k})^{1/2}}   + \frac{|y_{n}^{j} - y_{n}^{k} |}{(\lambda_{n, 2}^{j} \lambda_{n, 2}^{k})^{1/2}} \right.  \\
& \hspace{4cm} \left. + \frac{|t_{n}^{j} (\lambda_{n,1}^{j} \lambda_{n,2}^{j}) - t_{n}^{k} (\lambda_{n, 1}^{k} \lambda_{n, 2}^{k})|}{(\lambda_{n, 1}^{j} \lambda_{n, 2}^{j} \lambda_{n, 1}^{k} \lambda_{n, 2}^{k})^{1/2}}    \right]=\infty.
\end{align*}

\end{proposition}

The proof of Proposition \ref{supcritPD} of this follows by simple adaptation of the techniques in \cite[Section 4.4]{KV}, as originally introduced in \cite{Ker1}. We note that a similar statement also appears in works of Fanelli-Visciglia \cite{FV}.

\section{Profile decomposition for the mass-critical HLS}
\label{crit}

In this section, we focus on the mass-critical case. This case is more delicate for two reasons. First we need to account for the Galilean invariance symmetry in \eqref{1.4} and second, we cannot use a simple Sobolev estimate as in \eqref{Sob2} to fix the frequency scales. We follow closely the work in \cite[Section $4$]{KV} with a small variant in the use of modulation orthogonality and an additional argument for interactions of rectangles with skewed aspect ratios.

\subsection{A precised Strichartz inequality}

The main result in this section is the following proposition from which it is not hard to obtain a good profile decomposition. We need to introduce the norm
\begin{equation}\label{DefXp}
\Vert \phi\Vert_{X_p}:=\left(\sum_{R\in\mathcal{R}}\vert R\vert^{-\frac{p}{20}}\Vert \phi\mathfrak{1}_{R}\Vert_{L^\frac{20}{11}}^{p}\right)^\frac{1}{p}
\end{equation}
where $\mathcal{R}$ stands for the collection of all dyadic rectangles.    That is, rectangles with both sides parallel to an axis, of possibly different dyadic size, whose center is a multiple of the same dyadic numbers, given by the form
\begin{equation}
\label{Rdesc}
\begin{split}
\mathcal{R} &:= \{R_{k,n,p,m}\,:\,k, n, m, p \in \mathbb{Z}\},\\
R_{k,n,p,m} &:= \{  (x,y) :\,\, n-1 \leq 2^{-k}x \leq n+1,\,\,  m-1 \leq 2^{-p}y \leq m+1  \}.
\end{split}
\end{equation}
Note in particular that these spaces are nested: $X_p\subset X_q$ whenever $p\le q$.  The motivation for the space $L^\frac{20}{11}$ in \eqref{DefXp} can be motivated by the $X^q_p$ Strichartz estimate in Theorem $4.23$ from \cite{KV}.

\begin{proposition}\label{PD1}
Let $\phi\in C^\infty_c(\mathbf{R}^2)$, then, there holds that for all $p>2$,
\begin{equation}\label{NormIneq}
\Vert \phi\Vert_{X_p}\lesssim_p \Vert \phi\Vert_{L^2}
\end{equation}
and in addition, there exists $p>2$ such that
\begin{equation}\label{NormIneq2}
\Vert e^{it\partial_x \partial_y}\phi\Vert_{L^4_{x,y,t}}^4\lesssim \left(\sup_{R}\vert R\vert^{-\frac{1}{2}}\sup_{x,y,t}\vert e^{it\partial_x \partial_y}(P_R\phi)(x,y)\vert \right)^{\frac{4}{21}}\Vert \widehat{\phi}\Vert_{X_p}^{\frac{80}{21}}.
\end{equation}

\end{proposition}

We refer to \cite{RV} for a different proof of a slightly stronger estimate. Let us first recall the Whitney decomposition.

\begin{lemma}[Whitney decomposition]

There exists a tiling of the plane minus the diagonal
\begin{equation*}
\begin{split}
\mathbf{R}^2\setminus D&=\Cup I\times J,\qquad  D =\{(x,x),\quad x\in\mathbf{R}\},
\end{split}
\end{equation*}
made of dyadic intervals such that $\vert I\vert=\vert J\vert$ and
\begin{equation*}
\begin{split}
6\vert I\vert\le\hbox{dist}(I\times J,D)\le 24 \vert I\vert.
\end{split}
\end{equation*}

\end{lemma}

We will consider two independent Whitney decompositions of $\mathbf{R}\times\mathbf{R}$:
\begin{equation}\label{WD}
\begin{split}
1_{\{\mathbf{R}^2\times\mathbf{R}^2\setminus D\}}(\xi_1,\eta_1,\xi_2,\eta_2):=\sum_{I_1\sim I_2,\,\, J_1\sim J_2}\mathfrak{1}_{I_1}(\xi_1)\mathfrak{1}_{J_1}(\eta_1)\mathfrak{1}_{I_2}(\xi_2)\mathfrak{1}_{J_2}(\eta_2),
\end{split}
\end{equation}
where $I_i$ and $J_j$ are dyadic intervals of $\mathbf{R}$ and $\sim$ is an equivalence relation such that, for each fixed $I$, there are only finitely many $J$'s such that $I\sim J$, uniformly in $I$ (i.e. equivalence classes have bounded cardinality) and if $I\sim J$, then $\vert I\vert=\vert J \vert$ and $\hbox{dist}(I,J)\simeq \vert I\vert$. We also extend the equivalence relation to rectangles in the following fashion:
\begin{equation*}
I\times J \sim I'\times J' \qquad \text{if and only if} \qquad  I \sim I' \ \text{and} \ J \sim J'.
\end{equation*}

\medskip

We would like to follow the argument in \cite{KV} for the profile decomposition for the elliptic nonlinear Schr{\"o}dinger equation. However, it is at this point where we reach the main technical obstruction to doing this. Recall that to estimate the $L_{x,y,t}^{2}$ norm of $[e^{it \Delta} f]^{2}$, it was possible to utilize Plancherel's theorem, reducing the $L_{x,y,t}^{2}$ norm to an $l^{2}$ sum over pairs of Whitney squares.

This was because Plancherel's theorem in frequency turned the sum over all pairs of {\it equal area} squares to an $l^{2}$ sum over squares centered at different points in frequency space, and then Plancherel's theorem in time separated out pairs of squares with different area. Because there is only one square with a given area and center in space, this is enough. However, there are infinitely many rectangles with the same area and the same center. Thus, to reduce the $L_{x,y,t}^{2}$ norm of $[e^{it \partial_x\partial_y} f]^{2}$ to a $l^{2}$ sum over pairs of rectangles, that is rectangles whose sides obey the equivalence relation in both $x$ and $y$, it is necessary to deal with the off - diagonal terms, that is terms of the form

\begin{equation}
\| [e^{it \partial_x\partial_y} P_{R_{1}} f][e^{it \partial_x\partial_y} P_{R_{2}} f][\overline{e^{it \partial_x\partial_y} P_{R_{1}'} f}] [\overline{e^{it \partial_x\partial_y} P_{R_{2}'} f}] \|_{L_{x,y,t}^{1}},
\end{equation}  

\noindent where $R_{1} \sim R_{2}$ and $R_{1}' \sim R_{2}'$ are Whitney pairs of rectangles which have the same area, but very different dimensions in $x$ and $y$. In this case, Lemma \ref{BS1} gives a clue with regard to how to proceed, since it leads to the generalized result that

\begin{equation}
\| [e^{it \partial_x\partial_y} P_{R_{1}} f] [\overline{e^{it  \partial_x\partial_y} P_{R_{1}'} f}] \|_{L_{x,y,t}^{2}} \ll \| P_{R_{1}} f \|_{L^{2}_{x,y}} \| P_{R_{1}'} f \|_{L^{2}_{x,y}}.
\end{equation}

\noindent Thus, it may be possible to sum the off diagonal terms. We will not use Lemma $\ref{BS1}$ specifically, but we will use the idea that rectangles of the same area but very different dimensions have very weak bilinear interactions. 

Before we turn to the details, we first present the main orthogonality properties we will use. For simplicity of notation, given a dyadic rectangle $R$, let
\begin{equation*}
\widehat{\phi_R}(x,y):=\widehat{\phi}(x,y)\mathfrak{1}_R(x,y)\qquad\hbox{ and }\qquad u_R(x,y,t):=\left(e^{it\partial_x\partial_y}\phi_R\right)(x,y)
\end{equation*}
and set $u=e^{it\partial_x\partial_y}\phi$. Also we will consider rectangles $R_1=I_1\times J_1$, $R_2=I_2\times J_2$, $R_1'=I_1'\times J_1'$, $R_2'=I_2' \times J_2'$.

Proceeding with the above philosophy in mind, using \eqref{WD}, we have that

\begin{equation*}
\begin{split}
\Vert u\Vert_{L^4_{x,y,t}}^4&=\Vert u^2\Vert_{L^2_{x,y,t}}^2=\Vert \sum_{R_1\sim R_2}u_{R_1}\cdot u_{R_2}\Vert_{L^2_{x,y,t}}^2\\
&=\Vert \sum_{\Omega}\sum_{\substack{R_1\sim R_2,\\
\vert R_1\vert=\vert R_2\vert=\Omega}}u_{R_1}\cdot u_{R_2}\Vert_{L^2_{x,y,t}}^2=\Vert \sum_{\Omega} I_\Omega\Vert_{L^2_{x,y,t}}^2.
\end{split}
\end{equation*}

Using the polarization identity for a quadratic form,
\begin{equation*}
Q(x_1,y_1)+Q(x_2,y_2)=\frac{1}{2}\left[Q(x_1+x_2,y_1+y_2)+Q(x_1-x_2,y_1-y_2)\right],
\end{equation*}
 we compute that
\begin{equation*}
\begin{split}
e^{i\frac{t}{2}\xi\eta}\widehat{I_\Omega}&(\xi,\eta,t)=\sum_{\substack{R_1\sim R_2,\\
\vert R_1\vert=\vert R_2\vert=\Omega}}\int_{\mathbf{R}^4}\mathfrak{1}_{R_1}(\xi_1,\eta_1)\mathfrak{1}_{R_2}(\xi_2,\eta_2)e^{-i\frac{t}{2}(\xi_1-\xi_2)(\eta_1-\eta_2)}  \\
&\times \widehat{f}(\xi_1,\eta_1)\widehat{f}(\xi_2,\eta_2)\delta(\xi-\xi_1-\xi_2)\delta(\eta-\eta_1-\eta_2)d\xi_1 d\xi_2 d\eta_1 d\eta_2.
\end{split}
\end{equation*}
Now we observe that since
\begin{equation*}
\vert I_1\vert=\vert I_2\vert\simeq\hbox{dist}(I_1,I_2),\qquad \vert J_1\vert=\vert J_2\vert\simeq\hbox{dist}(J_1,J_2),
\end{equation*}
it holds that, on the support of integration,
\begin{equation*}
\vert (\xi_1-\xi_2)(\eta_1-\eta_2)\vert\simeq \vert I_1\vert\cdot\vert J_1\vert\simeq \Omega.
\end{equation*}
Therefore, we have the following orthogonality in time
\begin{equation*}
\Vert \sum_\Omega\widehat{I_\Omega}(\xi,\eta,\cdot)\Vert_{L^2_t}^2 =\Vert e^{i\frac{t}{2}\xi\eta}\sum_\Omega\widehat{I_\Omega}(\xi,\eta,\cdot)\Vert_{L^2_t}^2 \lesssim \sum_\Omega\Vert \widehat{I_\Omega}(\xi,\eta,\cdot)\Vert_{L^2_t}^2.
\end{equation*}

\medskip

To continue, we need to control $I_\Omega$ uniformly in $\Omega$. We write that
\begin{equation*}
\begin{split}
\mathcal{I}_\Omega
&=\left\Vert  \sum_{\substack{R_1\sim R_2,\\
\vert R_1\vert=\vert R_2\vert=\Omega}} u_{R_1}\cdot u_{R_2}\right\Vert_{L^2_{x,y,t}}^2\\
&=\sum_{\substack{R_1\sim R_2,\,\, R_1^\prime\sim R_2^\prime,\\
\vert R_1\vert=\vert R_2\vert=\vert R_1^\prime\vert=\vert R_2^\prime\vert=\Omega}}\int_{\mathbf{R}^3_{x,y,t}} u_{R_1}\cdot u_{R_2}\cdot\overline{u_{R_1^\prime}\cdot u_{R_2^\prime}}\,\, dx dy dt\\
&=\sum_{\substack{R_1\sim R_2,\,\, R_1^\prime\sim R_2^\prime\\
\vert R_1\vert=\vert R_2\vert=\vert R_1^\prime\vert=\vert R_2^\prime\vert=\Omega}}\mathcal{I}_{R_1\sim R_2,R_1^\prime\sim R_2^\prime}.
\end{split}
\end{equation*}
To any rectangle $R=I\times J$, we associate its center $c = (c_x, c_y)$ and its scales $\ell_x(R)=\vert I\vert$ and $\ell_y(R)=\vert J\vert=\Omega/\vert I\vert$. For 2 rectangles $R$ and $R^\prime$ of equal area, we define their relative discrepancy by
\begin{equation*}
\delta(R,R^\prime)=\min\{\ell_x(R)/\ell_x(R^\prime),\ell_y(R)/\ell_y(R^\prime)\}.
\end{equation*}
We want to decompose $\mathcal{I}_\Omega$ according to the discrepancy of $R_1=I_1\times J_1$ and $R_1^\prime=I_1^\prime\times J_1^\prime$. Using scaling relation \eqref{1.8}, we may assume that $\Omega=1$, $\ell_x(R_1)=\ell_x(R_2)=1$ and that $\ell_x(R_1^\prime)\le\ell_y(R_1^\prime)$, so that $R_1^\prime$ is a $\delta\times\delta^{-1}$ rectangle, where $\delta=\delta(R_1,R_1^\prime)$.

\medskip

We first notice that, if $\mathcal{I}_{R_1\sim R_2,R_1^\prime\sim R_2^\prime}\ne 0$, we must have that
\begin{equation}\label{OldOrtho}
\begin{split}
\vert c_x(R_1)-c_x(R_1^\prime)\vert&\,\,\lesssim\,\, \ell_x(R_1)+\ell_x(R_1'),\\
\vert c_y(R_1)-c_y(R_1^\prime)\vert&\,\,\lesssim\,\, \ell_y(R_1)+\ell_y (R_1').
\end{split}
\end{equation}
and therefore, for any fixed $R_1$ and $\delta\gtrsim 1$, there can be only a bounded number of choices for $R_1^\prime$, so that
\begin{equation*}
\begin{split}
\mathcal{I}_1&\lesssim \sum_{\substack{R_1\sim R_2,\\\vert R_1\vert=\vert R_2\vert=1}}\Vert u_{R_1}u_{R_2}\Vert_{L^2_{x,y,t}}^2.
\end{split}
\end{equation*}
At this stage, we are in a similar position as in the elliptic case and we may follow the proof in \cite[Section 4.4]{KV}. From now on, we will focus on the case $\delta\ll1$.

\medskip

In the case $\delta\ll1$, we may in fact strengthen \eqref{OldOrtho}. Indeed for $\mathcal{I}_{R_1\sim R_2,R_1^\prime\sim R_2^\prime}$ to be different from $0$, we must have that
\begin{equation}\label{NewOrtho}
\begin{split}
\vert c_x(R_1)-c_x(R_1^\prime)\vert&\,\, \simeq \,\, \ell_x(R_1)+\ell_x(R_1'),\\
\vert c_y(R_1)-c_y(R_1^\prime)\vert&\,\, \simeq \,\, \ell_y(R_1)+\ell_y (R_1').
\end{split}
\end{equation}
This follows from the fact that (say)
\begin{equation*}
\begin{split}
&c_x(R_1)+c_x(R_2)-c_x(R_1^\prime)-c_x(R_2^\prime)\\
=&2\left[c_x(R_1)-c_x(R_1^\prime)\right]-\left[c_x(R_1)-c_x(R_2)\right]+\left[c_x(R_1^\prime)-c_x(R_2^\prime)\right],
\end{split}
\end{equation*}
and the last bracket is bounded by $24\delta$, while the second to last is bounded below by $6$; however, for $\mathcal{I}_{R_1\sim R_2,R_1^\prime\sim R_2^\prime}$ to be nonzero, there must exists $(\xi_1,\xi_2,\xi_1^\prime,\xi_2^\prime)\in R_1\times R_2\times R_1^\prime\times R_2^\prime$ such that
\begin{equation*}
\begin{split}
&\xi_1+\xi_2-\xi_1^\prime-\xi_2^\prime=0\qquad\hbox{and}\\
&\vert (\xi_1+\xi_2-\xi_1^\prime-\xi_2^\prime)-(c_x(R_1)+c_x(R_2)-c_x(R_1^\prime)-c_x(R_2^\prime))\vert\le 2+2\delta.
\end{split}
\end{equation*}
We will keep note of this by writing $R_1\simeq R_1^\prime$ (or sometimes $c(R_1)\simeq c(R_1^\prime)$) whenever \eqref{NewOrtho} holds for rectangles of equal area.

Recall that $R_1^\prime$ is a $\delta\times\delta^{-1}$ rectangle; we can decompose all rectangles into $\delta\times 1$ rectangles. We may then partition
\begin{equation}\label{FinerDec}
\begin{split}
R_1&=\bigcup_{a=1}^{\delta^{-1}}I_{1,a}\times J_1=\bigcup_{a=1}^{\delta^{-1}}R_{1,a},\qquad R_2=\bigcup_{\widetilde{a}=1}^{\delta^{-1}} I_{2,\widetilde{a}}\times J_2=\bigcup_{\widetilde{a}=1}^{\delta^{-1}}R_{2,\widetilde{a}},\\
R_1^\prime&=\bigcup_{b=1}^{\delta^{-1}}I_1^\prime\times J_{1,b}^\prime=\bigcup_{b=1}^{\delta^{-1}}R_{1,b}^\prime,\qquad R_2^\prime=\bigcup_{\widetilde{b}=1}^{\delta^{-1}}I_2^\prime\times J_{2,\widetilde{b}}^\prime=\bigcup_{\widetilde{b}=1}^{\delta^{-1}}R_{2,\widetilde{b}}^\prime
\end{split}
\end{equation}
and by orthogonality, we see that
\begin{equation*}
\begin{split}
\mathcal{I}_{R_1\sim R_2,R_1^\prime\sim R_2^\prime}&=\sum_{a\sim\widetilde{a},\,\, b\sim\widetilde{b}}\mathcal{I}_{R_{1,a}\sim R_{2,\widetilde{a}},R_{1,b}^\prime\sim R_{2,\widetilde{b}}^\prime}
\end{split}
\end{equation*}
where 
\begin{equation*}
a\sim\widetilde{a} \qquad \text{if and only if} \qquad | c_x (R_{1,a}) + c_x(R_{2,\widetilde{a}})  - c_x (R_1') - c_x (R_2') | \lesssim  \delta
\end{equation*}
and comparably in $y$ for $b\sim \widetilde{b}$. Thus, for fixed $R_1$, $R_2$, $R_1^\prime$ $R_2^\prime$, this gives two equivalence relations with $O(\delta^{-1})$ equivalence classes of (uniformly) bounded cardinality.

And proceeding as in \eqref{NewOrtho}, we can easily see that
\begin{equation}\label{NewOrtho2}
\begin{split}
\vert c_x(R_{1,a})-c_x(R^\prime_{1,b})\vert\gtrsim 1,\qquad \vert c_y(R_{1,a})-c_y(R^\prime_{1,b})\vert\gtrsim \delta^{-1},\\
\vert c_x(R_{2,\widetilde{a}})-c_x(R^\prime_{2,\widetilde{b}})\vert\gtrsim 1,\qquad \vert c_y(R^\prime_{1,\widetilde{a}})-c_y(R^\prime_{2,\widetilde{b}})\vert\gtrsim\delta^{-1}.
\end{split}
\end{equation}

At this point, we have extracted all the orthogonality we need and we are ready to proceed with the proof of Proposition \ref{PD1}.

\subsection{Proof of \eqref{NormIneq2}}

Using rescaling, we may assume that
\begin{equation}\label{NormalAss}
1=\sup_{R}\vert R\vert^{-\frac{1}{2}}\Vert e^{it\partial_x \partial_y}\phi_R\Vert_{L^\infty_{x,y,t}}.
\end{equation}
From the considerations above, we obtain the expression
\begin{equation}\label{L4Sum1}
\begin{split}
\Vert e^{it\partial_x \partial_y}\phi\Vert_{L^4_{x,y,t}}^4&\lesssim \sum_{\Omega}\sum_{\substack{R_1\sim R_2,\, R_1^\prime\sim R_2^\prime,\\\vert R_1\vert=\vert R_2\vert=\vert R_1^\prime\vert=\vert R_2^\prime\vert=\Omega}}\mathcal{I}_{R_1\sim R_2,R_1^\prime\sim R_2^\prime},\\
\end{split}
\end{equation}
where the rectangles satisfy the condition \eqref{NewOrtho}. In addition, for fixed rectangles $R_1\sim R_2$, $R_1^\prime\sim R_2^\prime$ of equal area $\Omega$, let $\delta=\delta(R_1,R_1^\prime)$. As explained above, for fixed $\delta=\delta_0=O(1)$, we are in a position similar to the elliptic case and we may follow \cite{KV} to get
\begin{equation*}
\begin{split}
&\sum_{\Omega}\sum_{\substack{R_1\sim R_2,\, R_1^\prime\sim R_2^\prime,\\ \vert R_1\vert=\vert R_2\vert=\vert R_1^\prime\vert=\vert R_2^\prime\vert=\Omega,\\
\delta(R_1,R_1^\prime)=\delta_0}}\vert \mathcal{I}_{R_1\sim R_2,R_1^\prime\sim R_2^\prime}\vert\\
\lesssim& \sum_{\Omega}\sum_{\substack{R_1\sim R_2,\\ \vert R_1\vert=\vert R_2\vert=\Omega}}\Vert u_{R_1}u_{R_2}\Vert_{L^2_{x,y,t}}^2\\
\lesssim& \left(\sup_{R}\vert R\vert^{-\frac{1}{2}}\Vert u_R\Vert_{L^\infty_{x,y,t}}\right)^\frac{4}{21}\sum_{\substack{R_1\sim R_2}}\vert R_1\vert^{\frac{2}{21}}\Vert u_{R_1}u_{R_2}\Vert_{L^\frac{40}{21}_{x,y,t}}^\frac{40}{21}\\
\lesssim&\sum_{\substack{R_1\sim R_2}}\left\{\vert R_1\vert^{-\frac{1}{20}}\Vert \widehat{\phi_{R_1}}\Vert_{L^\frac{20}{11}_{x,y}}\cdot\vert R_2\vert^{-\frac{1}{20}}\Vert\widehat{\phi_{R_2}}\Vert_{L^\frac{20}{21}_{x,y,t}}\right\}^\frac{40}{21},
\end{split}
\end{equation*}
where we have used Cauchy-Schwarz in the first inequality, H\"older's inequality in the second and \eqref{NormalAss} together with Lemma \ref{LemBSE3} in the last inequality. 
This gives a bounded contribution as in \eqref{NormIneq2} for any $p\le 80/21$.

\medskip

We need to adjust the above scheme when $\delta\ll1$. In the following, we let
\begin{equation*}
T_{\ll1}:=\sum_{\delta\ll1}\sum_{\Omega}\sum_{\substack{R_1\sim R_2,\, R_1^\prime\sim R_2^\prime,\\ \vert R_1\vert=\vert R_2\vert=\vert R_1^\prime\vert=\vert R_2^\prime\vert=\Omega,\\
\delta(R_1,R_1^\prime)=\delta}}\vert \mathcal{I}_{R_1\sim R_2,R_1^\prime\sim R_2^\prime}\vert\\
\end{equation*}
and to conclude the proof of \eqref{NormIneq2}, we need to prove that, for some $p>2$,
\begin{equation}\label{L4Sum22}
T_{\ll1}\lesssim \Vert \phi\Vert_{X_p}^\frac{80}{21}.
\end{equation}
We can now use the finer decomposition \eqref{FinerDec} to write
\begin{equation*}
\begin{split}
\mathcal{I}_{R_1\sim R_2,R_1^\prime\sim R_2^\prime}&=\sum_{a\sim\widetilde{a},\, b\sim\widetilde{b}}\mathcal{I}_{R_{1,a}\sim R_{2,\widetilde{a}},R_{1,b}^\prime\sim R_{2,\widetilde{b}}^\prime}
\end{split}
\end{equation*}
where the new rectangles satisfy \eqref{NewOrtho2}. Using Cauchy-Schwartz, then H\"older's inequality with \eqref{NormalAss}, we have that
\begin{equation*}
\begin{split}
\vert \mathcal{I}_{R_{1,a}\sim R_{2,\widetilde{a}},R_{1,b}^\prime\sim R_{2,\widetilde{b}}^\prime}\vert&\lesssim \Vert u_{R_{1,a}}\cdot u_{R^\prime_{1,b}}\Vert_{L^2_{x,y,t}}\cdot\Vert u_{R_{2,\widetilde{a}}}\cdot u_{R^\prime_{2,\widetilde{b}}}\Vert_{L^2_{x,y,t}}\\
&\lesssim (\delta\Omega)^\frac{2}{21}\Vert u_{R_{1,a}}\cdot u_{R^\prime_{1,b}}\Vert_{L^\frac{40}{21}_{x,y,t}}^{\frac{20}{21}}\cdot\Vert u_{R_{2,\widetilde{a}}}\cdot u_{R^\prime_{2,\widetilde{b}}}\Vert_{L^\frac{40}{21}_{x,y,t}}^\frac{20}{21} .
\end{split}
\end{equation*}
Now, using Lemma \ref{LemBSE3} with \eqref{NewOrtho2}, we obtain that
\begin{equation*}
\begin{split}
&\vert \mathcal{I}_{R_{1,a}\sim R_{2,\widetilde{a}},R_{1,b}^\prime\sim R_{2,\widetilde{b}}^\prime}\vert\\
\lesssim &(\delta\Omega)^\frac{2}{21}\cdot (\delta^{-1}\Omega)^{-\frac{6}{21}}\Vert \widehat{\phi_{R_{1,a}}}\Vert_{L^{\frac{20}{11}}_{x,y}}^\frac{20}{21}\Vert \widehat{\phi_{R_{2,\widetilde{a}}}}\Vert_{L^{\frac{20}{11}}_{x,y}}^\frac{20}{21}\Vert \widehat{\phi_{R^\prime_{1,b}}}\Vert_{L^{\frac{20}{11}}_{x,y}}^\frac{20}{21}\Vert \widehat{\phi_{R^\prime_{2,\widetilde{b}}}}\Vert_{L^{\frac{20}{11}}_{x,y}}^\frac{20}{21}.
\end{split}
\end{equation*}
Since $20/11 <40/21$ and since for fixed $a$, there are only a bounded number $\widetilde{a}$ such that $a\sim \widetilde{a}$, we can sum over $a$ to get
\begin{equation*}
\begin{split}
\sum_{a\sim\widetilde{a}}\Vert \widehat{\phi_{R_{1,a}}}\Vert_{L^{\frac{20}{11}}_{x,y}}^\frac{20}{21}\Vert \widehat{\phi_{R_{2,\widetilde{a}}}}\Vert_{L^{\frac{20}{11}}_{x,y}}^\frac{20}{21}&
\lesssim \Vert \widehat{\phi_{R_1}}\Vert_{L^\frac{20}{11}_{x,y}}^\frac{20}{21}\Vert \widehat{\phi_{R_{2}}}\Vert_{L^\frac{20}{11}_{x,y}}^\frac{20}{21}
\end{split}
\end{equation*}
and similarly for $b$, so that
\begin{equation}\label{L4Sum2.0}
\begin{split}
\vert \mathcal{I}_{R_1\sim R_2,R_1^\prime\sim R_2^\prime}\vert&\lesssim \delta^{\frac{8}{21}}\Omega^{-\frac{4}{21}}\Vert \widehat{\phi_{R_1}}\Vert_{L^\frac{20}{11}_{x,y}}^\frac{20}{21}\Vert \widehat{\phi_{R_{2}}}\Vert_{L^\frac{20}{11}_{x,y}}^\frac{20}{21}\Vert \widehat{\phi_{R_1^\prime}}\Vert_{L^\frac{20}{11}_{x,y}}^\frac{20}{21}\Vert \widehat{\phi_{R_{2}^\prime}}\Vert_{L^\frac{20}{11}_{x,y}}^\frac{20}{21}.
\end{split}
\end{equation}
In addition, for rectangles of fixed areas and sizes $\vert R_1\vert=\vert R_2\vert=\vert R_1^\prime\vert=\vert R_2^\prime\vert$, $\ell_x(R_1)=\ell_x(R_2)$, $\ell_x(R_1^\prime)=\ell_x(R_2^\prime)$ also satisfying \eqref{NewOrtho}, we may use Cauchy Schwartz in the summation over the centers to get
\begin{equation*}
\begin{split}
\sum_{\substack{R_1\sim R_2,\\ R_1^\prime\sim R_2^\prime}}\Vert \widehat{\phi_{R_1}}\Vert_{L^\frac{20}{11}_{x,y}}^\frac{20}{21}\Vert \widehat{\phi_{R_{2}}}\Vert_{L^\frac{20}{11}_{x,y}}^\frac{20}{21}\Vert \widehat{\phi_{R_1^\prime}}\Vert_{L^\frac{20}{11}_{x,y}}^\frac{20}{21}\Vert \widehat{\phi_{R_{2}^\prime}}\Vert_{L^\frac{20}{11}_{x,y}}^\frac{20}{21}&
\lesssim \sum_{R_1\simeq R_1^\prime}\Vert \widehat{\phi_{R_1}}\Vert_{L^\frac{20}{11}_{x,y}}^\frac{40}{21}\Vert \widehat{\phi_{R_1^\prime}}\Vert_{L^\frac{20}{11}_{x,y}}^\frac{40}{21}
\end{split}
\end{equation*}
where the sum is taken over all rectangles $R_1\simeq R_1^\prime$ of the given sizes satisfying \eqref{NewOrtho}.

\medskip

We can now get back to \eqref{L4Sum22} and use \eqref{L4Sum2.0} and the inequality above to get
\begin{equation*}
\begin{split}
T_{\ll1}&\lesssim\sum_{\Omega}\sum_A\sum_{\delta\le 1}\delta^{\frac{8}{21}}\Omega^{-\frac{4}{21}}\cdot \sum_{c_1\simeq c_1^\prime}\Vert \widehat{\phi_{R_1}}\Vert_{L^\frac{20}{11}_{x,y}}^\frac{40}{21}\Vert \widehat{\phi_{R_1^\prime}}\Vert_{L^\frac{20}{11}_{x,y}}^\frac{40}{21},
\end{split}
\end{equation*}
where we have parameterized the lengths of the rectangles by $\Omega=\vert R_1\vert=\vert R_1^\prime\vert$, $A=\ell_x(R_1)$ and $\delta=\ell_x(R_1^\prime)/\ell_x(R_1)$, and their centers by $c_1$, $c_1^\prime$.

Now for any $p > 2$ choose $0 < \theta(p) < 1$ such that
\begin{equation}\label{1.5m}
\frac{2 \theta}{p} + \frac{1 - \theta}{p} = \frac{21}{40}.
\end{equation}
and observe that $\theta(p) \searrow \frac{1}{20}$ as $p \searrow 2$. Then by interpolation,
\begin{equation*}
\begin{split}
&\sum_{\substack{\Omega,A,\\
c_1\simeq c_1^\prime}}\Omega^{-\frac{4}{21}}\Vert \widehat{\phi_{R_1}}\Vert_{L^\frac{20}{11}_{x,y}}^\frac{40}{21}\Vert \widehat{\phi_{R_1^\prime}}\Vert_{L^\frac{20}{11}_{x,y}}^\frac{40}{21}\\
\lesssim&\left(\sum_{\substack{\Omega,A,\\ c_1\simeq c_1^\prime}}\Omega^{-\frac{p}{10}}\Vert \widehat{\phi_{R_1}}\Vert_{L^\frac{20}{11}_{x,y}}^{p}\Vert \widehat{\phi_{R_1^\prime}}\Vert_{L^\frac{20}{11}_{x,y}}^{p}\right)^{\frac{40}{21}\frac{1-\theta}{p}}\left(\sum_{\substack{\Omega,A,\\ c_1\simeq c_1^\prime}}\Omega^{-\frac{p}{20}}\Vert \widehat{\phi_{R_1}}\Vert_{L^\frac{20}{11}_{x,y}}^\frac{p}{2}\Vert \widehat{\phi_{R_1^\prime}}\Vert_{L^\frac{20}{11}_{x,y}}^\frac{p}{2}\right)^{\frac{40}{21}\frac{2\theta}{p}}.
\end{split}
\end{equation*}
Now, on the one hand, we observe that for a fixed choice of scales ($\Omega$, $A$ and $\delta$)  and for each fixed $c_1$, there are at most $O(\delta^{-1})$ choices of $c_1^\prime$ satisfying \eqref{NewOrtho} so we obtain that
\begin{equation}\label{L4Sum2}
\begin{split}
\sum_{\substack{\Omega,A,\\ c_1\simeq c_1^\prime}}\Omega^{-\frac{p}{20}}\Vert \widehat{\phi_{R_1}}\Vert_{L^\frac{20}{11}_{x,y}}^\frac{p}{2}\Vert \widehat{\phi_{R_1^\prime}}\Vert_{L^\frac{20}{11}_{x,y}}^\frac{p}{2}&\lesssim \delta^{-1}\sum_{\substack{\Omega,A,c_1}}\Omega^{-\frac{p}{20}}\Vert \widehat{\phi_{R_1}}\Vert_{L^\frac{20}{11}_{x,y}}^p
\end{split}
\end{equation}
and the other sum can be handled in an easier way: using H\"older's inequality and forgetting about the relationship $c_1\simeq c_1^\prime$, we obtain that
\begin{equation}\label{L4Sum3}
\begin{split}
&\sum_{\Omega}\sum_A\sum_{c_1\simeq c_1^\prime}\Omega^{-\frac{p}{10}}\Vert \widehat{\phi_{R_1}}\Vert_{L^\frac{20}{11}_{x,y}}^{p}\Vert \widehat{\phi_{R_1^\prime}}\Vert_{L^\frac{20}{11}_{x,y}}^{p}\\
\lesssim&\sum_{\Omega}\sum_A\left(\sum_{c_1}\left\{\Omega^{-\frac{1}{20}}\Vert \widehat{\phi_{R_1}}\Vert_{L^\frac{20}{11}_{x,y}}\right\}^{p}\right)\cdot\left(\sum_{c_1^\prime}\left\{\Omega^{-\frac{1}{20}}\Vert \widehat{\phi_{R_1^\prime}}\Vert_{L^\frac{20}{11}_{x,y}}\right\}^{p}\right)\\
\lesssim& \left(\sum_{R}\vert R\vert^{-\frac{p}{20}}\Vert \widehat{\phi_R}\Vert_{L^{\frac{20}{11}}_{x,y}}^p\right)^2.
\end{split}
\end{equation}
Recall the definition \eqref{DefXp}. Combining \eqref{L4Sum2} and \eqref{L4Sum3}, we obtain
\begin{equation*}
\begin{split}
T_{\ll1}&\lesssim \sum_{\delta\le 1}\delta^{\frac{8}{21}}\cdot\left(\delta^{-1}\Vert \widehat{\phi}\Vert_{X_p}^p\right)^{\frac{80}{21}\frac{\theta}{p}}\left(\Vert \widehat{\phi}\Vert_{X_p}^{2p}\right)^{\frac{40}{21}\frac{1-\theta}{p}}\lesssim \sum_{\delta\le 1}\delta^{\frac{8}{21}(1-\frac{10\theta}{p})}\Vert\widehat{\phi}\Vert_{X_p}^{\frac{80}{21}}
\end{split}
\end{equation*}
and this is summable in $\delta$ for $2<p<40/17$ small enough. The proof of \eqref{NormIneq2} is thus complete and it remains to prove \eqref{NormIneq} which we now turn to.

\subsection{Proof of \eqref{NormIneq}}

We first state and prove the following simple result we will need in the proof.

\begin{lemma}\label{ModelLem}
Let $\mathcal{D}$ denote the set of dyadic intervals (on $\mathbf{R}$) and let $p>2$. For any $g\in C^\infty_c(\mathbf{R})$, there holds that
\begin{equation}
\sum_{I\in\mathcal{D}}\vert I\vert^{-\frac{p}{20}}\Vert g\mathfrak{1}_I\Vert_{L^\frac{20}{11}_x}^{p}\lesssim \Vert g\Vert_{L^2_x}^p.
\end{equation}
\end{lemma}

\begin{proof}[Proof of Lemma \ref{ModelLem}]
We may assume that $\Vert g\Vert_{L^2_x}=1$. For fixed $A$, we let $\mathcal{D}_A$ denote the set of dyadic intervals of length $A$ and we decompose
\begin{equation*}
\begin{split}
g=g^++g^-,\qquad g^+(x)=g(x)\mathfrak{1}_{\{\vert g(x)\vert>A^{-\frac{1}{2}}\}},\qquad g^-(x)=g(x)\mathfrak{1}_{\{\vert g(x)\vert\le A^{-\frac{1}{2}}\}}.
\end{split}
\end{equation*}
On the one hand, using that $\ell^{\frac{20}{11}}\subset\ell^p$,
\begin{equation*}
\begin{split}
\sum_A\sum_{I\in\mathcal{D}_A}\vert I\vert^{-\frac{p}{20}}\Vert g^+\mathfrak{1}_I\Vert_{L^\frac{20}{11}_x}^p&\lesssim \left(\sum_AA^{-\frac{1}{11}}\sum_{I\in\mathcal{D}_A}\Vert g^+\mathfrak{1}_I\Vert_{L^\frac{20}{11}_x}^\frac{20}{11}\right)^{\frac{11 }{20}p}\\
&\lesssim \left(\sum_AA^{-\frac{1}{11}}\int_{\mathbf{R}}\vert g\vert^\frac{20}{11}\mathfrak{1}_{\{\vert g(x)\vert>A^{-\frac{1}{2}}\}}dx\right)^{\frac{11 }{20}p}\\
&\lesssim \left(\int_{\mathbf{R}}\vert g\vert^\frac{20}{11}\cdot\left(\sum_{A>\vert g(x)\vert^{-2}}A^{-\frac{1}{11}}\right)dx\right)^{\frac{11 }{20}p}\lesssim 1,
\end{split}
\end{equation*}
while, for the other sum, we use H\"older's inequality to get
\begin{equation*}
\begin{split}
\sum_A\sum_{I\in\mathcal{D}_A} A^{-\frac{p}{20}}\Vert g^-\mathfrak{1}_I\Vert_{L^\frac{20}{11}_x}^p
&\lesssim \sum_A\sum_{I\in\mathcal{D}_A} A^{-\frac{p}{20}}\Vert g^-\mathfrak{1}_I\Vert_{L^p}^p\cdot \vert I\vert^{(\frac{11}{20}-\frac{1}{p})p}\\
&\lesssim \int_{\mathbf{R}}\vert g(x)\vert^p \cdot\sum_{\{A<\vert g(x)\vert^{-2}\}}A^{\frac{p-2}{2}}dx\\
&\lesssim \int_{\mathbf{R}}\vert g(x)\vert^2dx\lesssim 1
\end{split}
\end{equation*}
and the proof is complete.
\end{proof}

Now, we proceed to prove \eqref{NormIneq}.

\begin{proof}[Proof of \eqref{NormIneq}] 
Recall $\mathcal{D}$ stand for the set of dyadic intervals and $\mathcal{D}_A$ for the set of dyadic intervals of length $A$. We want to prove that
\begin{equation*}
\begin{split}
\sum_{I\in\mathcal{D}}\vert I\vert^{-\frac{p}{20}}\sum_{J\in\mathcal{D}}\vert J\vert^{-\frac{p}{20}}\Vert f \mathfrak{1}_{I\times J} \Vert_{L^\frac{20}{11}_{x,y}}^p \lesssim \Vert f\Vert_{L^2_{x,y}}^p.
\end{split}
\end{equation*}
We claim that, for any fixed interval $I$,
\begin{equation}\label{TnterClaim1}
\sum_{J\in\mathcal{D}}\vert J\vert^{-\frac{p}{20}}\Vert f\mathfrak{1}_{I\times J}\Vert_{L^\frac{20}{11}_{x,y}}^p \lesssim \Vert f\mathfrak{1}_{I\times\mathbf{R}}\Vert_{L^\frac{20}{11}_xL^2_y}^p.
\end{equation}
Once this is proved, we may simply apply Lemma \ref{ModelLem} to the function
\begin{equation*}
g(x):=\Vert f(x,\cdot)\Vert_{L^2_y}
\end{equation*}
to finish the proof.

\medskip

From now on $I$ denotes a fixed interval and $f$ is a function supported on $\{x\in I\}$, i.e. $f=f\mathfrak{1}_{I\times\mathbf{R}}$. The proof of \eqref{TnterClaim1} is a small variation on the proof of Lemma \ref{ModelLem}. Fix a dyadic number $B$ and let
\begin{equation*}
c_B=c_B(x)=B^{-\frac{1}{2}}\Vert f(x,\cdot)\Vert_{L^2_y}
\end{equation*}
and decompose accordingly\footnote{Note that $f(x,y)=0$ whenever $\Vert f(x,\cdot)\Vert_{L^2_y}=0$, so that $c_B(x)>0$ on the support of $f^+$.}
\begin{equation*}
f=f^++f^-,\quad f^+=f\mathfrak{1}_{\{\vert f(x,y)\vert>c_B(x)\}},\quad f^-=f\mathfrak{1}_{\{\vert f(x,y)\vert\le c_B(x)\}}.
\end{equation*}
We then compute that
\begin{equation*}
\begin{split}
& \sum_B\sum_{J\in\mathcal{D}_B} B^{-\frac{p}{20}}\Vert f^+\mathfrak{1}_{I\times J}\Vert_{L^\frac{20}{11}_{x,y}}^p  \lesssim\left(\sum_B\sum_{J\in\mathcal{D}_B} B^{-\frac{1}{11}}\Vert f^+\mathfrak{1}_{I\times J}\Vert_{L^\frac{20}{11}_{x,y}}^{\frac{20}{11}}\right)^{\frac{11 }{20}p}  \\
& \hspace{2cm} \lesssim\left(\sum_B B^{-\frac{1}{11}}\int_{I_x}\int_{\mathbf{R}_y} \vert f^+\vert^\frac{20}{11} dy\, dx\right)^{\frac{11 }{20}p}\\
& \hspace{2cm}  \lesssim\left(\int_{I_x}\int_{\mathbf{R}_y} \vert f^+\vert^\frac{20}{11}\cdot \sum_{\{B:\vert f(x,y)\vert\ge c_B(x)\}}B^{-\frac{1}{11}} dy\, dx\right)^{\frac{11}{20}p}\\
& \hspace{2cm}  \lesssim\left(\int_{I_x}\int_{\mathbf{R}_y} \vert f(x,y)\vert^\frac{20}{11}\cdot \left(\frac{ \vert f(x,y)\vert}{\Vert f(x,\cdot)\Vert_{L^2_y}}\right)^\frac{2}{11} dy\, dx\right)^{\frac{11}{20}p}\\
& \hspace{2cm} \lesssim \left(\int_{I_x}\Vert f(x,\cdot)\Vert_{L^2_{y}}^{-\frac{2}{11}}\int_{\mathbf{R}_y}\vert f(x,y)\vert^2dy\, dx\right)^{\frac{11}{20}p}\\
 & \hspace{2cm}  \lesssim \Vert f\Vert_{L^\frac{20}{11}_xL^2_y}^p ,
\end{split}
\end{equation*}
 in the penultimate line, we note that though there is a negative power of the $L^2_y$ norm, the product of the two quantities is well-defined, especially as we can assume $f \in C^\infty_c$.  Also, we have used the embedding $\ell^1 \subset  \ell^{\frac{11}{20}p}$ in the first inequality, the fact that dyadic intervals of a fixed length tile $\mathbf{R}$ in the second inequality, and we have summed a geometric series in the fourth inequality. Now for the second part, we compute that
\begin{equation*}
\begin{split}
\sum_B\sum_{J\in\mathcal{D}_B} &B^{-\frac{p}{20}}\Vert f^-\mathfrak{1}_{I\times J}\Vert_{L^\frac{20}{11}_{x,y}}^p  \lesssim \sum_B\sum_{J\in\mathcal{D}_B} B^{-\frac{p}{20}}B^{p(\frac{11}{20}-\frac{1}{p})}\Vert f^-\mathfrak{1}_{I\times J}\Vert_{L^p_y(J:L^\frac{20}{11}_x(I))}^p\\
&\lesssim \sum_B B^{ \frac{p -2 }{2}  }\int_{\mathbf{R}_y}\left(\int_{I_x} \vert f^-(x,y)\vert^\frac{20}{11}dx\right)^{\frac{11}{20}p}dy\\
&\lesssim \int_{\mathbf{R}_y}\sum_B \left(B^{\frac{p -2 }{2}\frac{20}{11}\frac{1}{p}}\int_{I_x} \vert f^-(x,y)\vert^\frac{20}{11}dx\right)^{\frac{11}{20}p}dy\\
&\lesssim \int_{\mathbf{R}_y}\left(\int_{I_x} \sum_{B}B^{\frac{p -2 }{p}\frac{10}{11}}\vert f^-(x,y)\vert^\frac{20}{11}dx\right)^{\frac{11}{20}p}dy,
\end{split}
\end{equation*}
where we have used H\"older's inequality in the first line and the inclusion $\ell^1\subset\ell^{\frac{11}{20}p}$ in the fourth line. Now, since $f^-$ is supported where
\begin{equation*}
B\le\left(\frac{\Vert f(x,\cdot)\Vert_{L^2_y}}{\vert f(x,y)\vert}\right)^2,
\end{equation*}
summing in $B$ gives
\begin{equation*}
\begin{split}
&\sum_B\sum_{J\in\mathcal{D}_B} B^{-\frac{p}{20}}\Vert f^-\mathfrak{1}_{I\times J}\Vert_{L^\frac{20}{11}_{x,y}}^p\\
\lesssim& \int_{\mathbf{R}_y}\left(\int_{I_x} \Vert f(x,\cdot)\Vert_{L^2_y}^{\frac{20}{11}\frac{p-2}{p}}\vert f^-(x,y)\vert^{\frac{20}{11}\frac{2}{p}}dx\right)^{\frac{11}{20}p}dy.
\end{split}
\end{equation*}
Using Minkowski inequality on the function
\begin{equation*}
h(x,y)=\Vert f(x,\cdot)\Vert_{L^2_y}^{\frac{20}{11}\frac{p-2}{p}}\vert f^-(x,y)\vert^{\frac{20}{11}\frac{2}{p}},
\end{equation*}
we obtain
\begin{equation*}
\begin{split}
&\sum_B\sum_{J\in\mathcal{D}_B} B^{-\frac{p}{20}}\Vert f^-\mathfrak{1}_{I\times J}\Vert_{L^\frac{20}{11}_{x,y}}^p\\
\lesssim& \left(\int_{I_x}\left(\int_{\mathbf{R}_y}h^{\frac{11}{20}p}dy\right)^{\frac{20}{11}\frac{1}{p}}dx\right)^{\frac{11}{20}p}\\
\lesssim&  \left(\int_{I_x}\left(\int_{\mathbf{R}_y}\vert f(x,y)\vert^2dy\right)^{\frac{20}{11}\frac{1}{p}}\Vert f(x,\cdot)\Vert_{L^2_y}^{\frac{20}{11}\frac{p-2}{p}}  dx\right)^{\frac{11}{20}p}\\
\lesssim& \left(\int_{I_x}\Vert f(x,\cdot)\Vert_{L^2_y}^\frac{20}{11}dx\right)^{\frac{11}{20}p},
\end{split}
\end{equation*}
which proves \eqref{TnterClaim1}. Thus the proof is complete.

\end{proof}

\vspace{.5cm}

\section{The Profile Decomposition and Applications}\label{SecMBUS}

The profile decomposition then follows from Proposition $\ref{PD1}$ in the usual way following the techniques in the proof of Theorems $4.25$ (the Inverse Strichartz Inequality) and $4.26$ (Mass Critical Profile Decomposition) from \cite{KV}, for instance.  We note that it is the proof of the Inverse Strichartz Inequality that requires the local smoothing estimates as in Lemma \ref{LocSmoothing} to establish pointwise a.e. convergence of profiles to an element of $L^2_{x,y}$ through compactness considerations, otherwise the proof follows mutatis mutandis.  Once the Inverse Strichartz Inequality is established, the proof of the Profile Decomposition follows verbatim.  

Suppose $g_{n}^{j} = g(x_{n}^{j}, y_{n}^{j}, \lambda_{1, n}^{j}, \lambda_{2, n}^{j}, \xi_{n}^{j})$ is the group whose action on functions is given by
\begin{align*}
& (g_{n}^{j})^{-1} f = (\lambda_{n, 1}^{j} \lambda_{n, 2}^{j})^{1/2} e^{-i \xi_{n, 1}^{j} (\lambda_{n, 1}^{j} x + x_{n}^{j})} e^{-i \xi_{n, 2}^{j} (\lambda_{n, 2}^{j} x + y_{n}^{j})} \\
& \hspace{2.5cm}  \times [f_{n}](\lambda_{n, 1}^{j} x + x_{n}^{j}, \lambda_{n, 2}^{j} y + y_{n}^{j}).
\end{align*}

The profile decomposition gives the following.

\begin{theorem}\label{ProfileDecTheorem}
Let $\| u_{n} \|_{L^{2}_{x,y}(\mathbf{R}^{2})} \leq A$ be a sequence that is bounded $L^{2}_{x,y}(\mathbf{R}^{2})$. Then possibly after passing to a subsequence, for any $1 \leq j < \infty$ there exist $\phi^{j} \in L^{2}_{x,y}(\mathbf{R}^{2})$, $(t_{n}^{j}, x_{n}^{j}, y_{n}^{j}) \in \mathbf{R}^{3}$, $\xi_{n}^{j} \in \mathbf{R}^{2}$, $\lambda_{n, 1}^{j}$, $\lambda_{n, 2}^{j} \in (0, \infty)$ such that for any $J$,

\begin{equation}
u_{n} = \sum_{j = 1}^{J} g_{n}^{j} e^{i t_{n}^{j} \partial_{x} \partial_y} \phi^{j} + w_{n}^{J},
\end{equation}

\begin{equation}\label{strichartz}
\lim_{J \rightarrow \infty} \limsup_{n \rightarrow \infty} \| e^{it \partial_x \partial_y  } w_{n}^{J} \|_{L_{x,y,t}^{4}} = 0,
\end{equation}

\noindent such that for any $1 \leq j \leq J$,

\begin{equation}
e^{-i t_{n}^{j} \partial_x \partial_y} (g_{n}^{j})^{-1} w_{n}^{J} \rightharpoonup 0,
\end{equation}

\noindent weakly in $L^{2}_{x,y}(\mathbf{R}^{2})$, 

\begin{equation}\label{decouple}
\lim_{n \rightarrow \infty} \left(\| u_{n} \|_{L^{2}_{x,y}}^{2} - \sum_{j = 1}^{J} \| \phi^{j} \|_{L^{2}_{x,y}}^{2} - \| w_{n}^{J} \|_{L^{2}_{x,y}}^{2}\right) = 0,
\end{equation}

\noindent and for any $j \neq k$,

\begin{equation}\label{AsympOrtho}
\begin{split}
& \lim_{n \rightarrow \infty}\left[ \left\vert\ln \left( \frac{\lambda_{n, 1}^{j}}{\lambda_{n, 1}^{k}} \right)\right\vert + \left\vert \ln \left( \frac{\lambda_{n, 2}^{j}}{\lambda_{n, 2}^{k}}  \right)\right\vert 
+ \frac{|t_{n}^{j} (\lambda_{n,1}^{j} \lambda_{n,2}^{j}) - t_{n}^{k} (\lambda_{n, 1}^{k} \lambda_{n, 2}^{k})|}{(\lambda_{n, 1}^{j} \lambda_{n, 2}^{j} \lambda_{n, 1}^{k} \lambda_{n, 2}^{k})^{1/2}}  \right. \\ 
& \hspace{1cm} + (\lambda_{n, 1}^{j} \lambda_{n, 1}^{k})^{1/2} |\xi_{n, 1}^{j} - \xi_{n, 1}^{k}| 
+ (\lambda_{n, 2}^{j} \lambda_{n, 2}^{k})^{1/2} |\xi_{n, 2}^{j} - \xi_{n, 2}^{k}|\\
& \hspace{2cm} + \frac{|x_{n}^{j} - x_{n}^{k} - 2 t_{n}^{j} (\lambda_{n, 1}^{j} \lambda_{n, 2}^{j})(\xi_{n, 1}^{j} - \xi_{n, 1}^{k})|}{(\lambda_{n, 1}^{j} \lambda_{n, 1}^{k})^{1/2}}  \\
& \left. \hspace{2cm} + \frac{|y_{n}^{j} - y_{n}^{k} - 2 t_{n}^{j} (\lambda_{n, 1}^{j} \lambda_{n, 2}^{j})(\xi_{n, 2}^{j} - \xi_{n, 2}^{k})|}{(\lambda_{n, 2}^{j} \lambda_{n, 2}^{k})^{1/2}} \right]   = \infty.
\end{split}
\end{equation}

\end{theorem}

\subsection{Minimal mass blow-up solutions}

As an application of the profile decomposition, we turn to a calculation that for instance originated in \cite{Ker1,MerleVega}.  Namely we construct a minimal mass solution to \eqref{1.1} which is a solution $u$ of minimal mass such that there exists a time $T^*$ such that
\begin{equation*}
\int_{-T^*}^{T^*} \int_{\mathbf{R}^2_{x,y}} |u |^4 \, dx dy dt= + \infty.
\end{equation*}
In other words, it is a solution of least mass for which the small data global argument fails. 

It turns out that if $u$ is a minimal mass blowup solution to $\eqref{1.1}$ then $u$ lies in a compact subset of $L^{2}_{x,y}(\mathbf{R}^{2})$ modulo the symmetry group $g$; more precisely, following \cite[Chapter $5$, Theorem $5.2$]{KV}, we can establish the following theorem.

\begin{theorem}
Suppose $u$ is a minimal mass blowup solution to $(\ref{1.1})$ on a maximal time interval $I$ that blows up in both time directions.  That is, $I$ is an open interval and for any $t_{0} \in I$,

\begin{equation}\label{InfiniteNorm}
\int_{t_{0}}^{\sup(I)} \int |u(x,y,t)|^{4} dx dy dt, \hspace{5mm} \int_{\inf(I)}^{t_{0}} \int |u(x,y,t)|^{4} dx dy dt = \infty.
\end{equation}

\noindent Then there exist $\lambda_{1}, \lambda_{2} : I \rightarrow (0, \infty)$, $\widetilde{\xi} : I \rightarrow \mathbf{R}^{2}$, $\widetilde{x}, \widetilde{y} : I \rightarrow \mathbf{R}$, such that for any $\eta > 0$ there exists $C(\eta) < \infty$ such that

\begin{equation}
\aligned
\int_{|x - \widetilde{x}(t)| > \frac{C(\eta)}{\lambda_{1}(t)}} |u(x,y,t)|^{2} dx dy + \int_{|y - \widetilde{y}(t)| > \frac{C(\eta)}{\lambda_{2}(t)}} |u(x,y,t)|^{2} dx dy \\ + \int_{|\xi_{1} - \widetilde{\xi}_{1}(t)| > C(\eta) \lambda_{1}(t)} |\hat{u}(\xi,t)|^{2} d\xi + \int_{|\xi_{2} - \widetilde{\xi}_{2}(t)| > C(\eta) \lambda_{2}(t)} |\hat{u}(\xi,t)|^{2} d\xi < \eta.
\endaligned
\end{equation}
\end{theorem}

\begin{proof} Take a sequence $t_{n} \in I$. Then conservation of mass implies that after passing to a subsequence we may make a profile decomposition of $u(t_{n}) = u_{n}$. If there exists $j$ such that, along a subsequence, $t_{n}^{j} \rightarrow \pm \infty$, say $t_{n}^{j} \rightarrow \infty$, then

\begin{equation}
\lim_{n \rightarrow \infty} \| e^{it \partial_x\partial_y} (g_{n}^{j} e^{i t_{n}^{j} \partial_x \partial_y} \phi^{j}) \|_{L_{x,y,t}^{4}([0, \infty) \times \mathbf{R}^{2})} = 0,
\end{equation}

\noindent so combining perturbative arguments, $(\ref{decouple})$, and the fact that $u$ is a blowup solution with minimal mass then $u$ scatters forward in time to a free solution. Thus, we may assume that for each $j$, $t_{n}^{j}$ converges to some $t^{j} \in \mathbf{R}$. Then taking $e^{it^{j} \partial_x \partial_y} \phi^{j}$ to be the new $\phi^{j}$, we may assume that each $t_{n}^{j} = 0$.

\noindent Now suppose that

\begin{equation}
\sup_{j} \| \phi^{j} \|_{L^{2}_{x,y}(\mathbf{R}^{2})} < \| u(t) \|_{L^{2}_{x,y}(\mathbf{R}^{2})}.
\end{equation}

\noindent Then if $v^{j}$ is the solution to $\eqref{1.1}$ with initial data $\phi^{j}$, since $\| u(t) \|_{L^{2}}$ is the minimal mass for blowup to occur, each $v^{j}$ scatters both forward and backward in time, with

\begin{equation}\label{bound}
 \| v^{j} \|_{L_{x,y,t}^{4}(\mathbf{R} \times \mathbf{R}^{2})}^2 \lesssim \Vert \phi^j\Vert_{L^2_{x,y}}^2 < \infty,\qquad  \hbox{uniformly in } j.
\end{equation}

\noindent Then if $v_{n}^{j}$ is the solution to $\eqref{1.1}$ with initial data $g_{n}^{j} \phi^{j}$, $v_{n}^{j} = g_{n}^{j} (v^{j}( (\lambda_{n,1}^j\lambda_{n,2}^j)^{-1}  t))$.
We note that, for $v$ either a profile $v_n^\ell$ or the remainder $w_n^J$,
\begin{equation}\label{Convto01}
\| v_{n}^{j} v_{n}^{k} v \|_{L_{x,y,t}^{\frac43}}  \le \Vert v_n^jv_n^k\Vert_{L^2_{x,y,t}}\Vert v \Vert_{L^4_{x,y,t}}.
\end{equation}
In addition, $\Vert v\Vert_{L^4_{x,y,t}}$ remains bounded either by \eqref{bound} (for $v_n^\ell$) or as a consequence of the small data theory and \eqref{strichartz} (for $w_n^J$).

By approximation by compactly supported functions, it is easy to see that, if $j\ne k$,
\begin{equation}\label{Convto02}
\Vert v_n^jv_n^k\Vert_{L^2_{x,y,t}}\to0
\end{equation}
when $n \rightarrow \infty$ as a consequence of \eqref{AsympOrtho}.

As a result, using simple perturbation theory, we obtain that, for $J$ large enough,
\begin{equation*}
\begin{split}
\Vert u(t_n+t)-\sum_{j=1}^Jv_n^j(t)\Vert_{L^4_{x,y,t}}\lesssim 1
\end{split}
\end{equation*}
and using again \eqref{Convto01}-\eqref{Convto02}, we obtain that
\begin{equation*}
\begin{split}
\Vert \sum_{j=1}^Jv_n^j(t)\Vert_{L^4_{x,y,t}}^4\lesssim \sum_{j=1}^J\Vert v_n^j\Vert_{L^4_{x,y,t}}^4\lesssim \sum_{j=1}^J\Vert \phi^j\Vert_{L^2_{x,y}}^2<\infty.
\end{split}
\end{equation*}
which, together with \eqref{decouple} contradicts \eqref{InfiniteNorm}.

Thus, after reordering we should have $\| \phi^{1} \|_{L^{2}_{x,y}} = \| u(t) \|_{L^{2}_{x,y}}$ and $\phi^{j} = 0$ for any $j \geq 2$. But this holds if and only if $u(t)$ lies in a set $G K$, where $G$ is the group generated by $g_{n}^{j}$ and $K$ is a compact set in $L^{2}$. This completes the proof of the theorem.
\end{proof}

\appendix

\section{Extremizers for Strichartz Estimates for \eqref{HLS}}
\label{extreme}

We end with a few considerations on extremizers for the Strichartz inequality \eqref{LinStric}, that is functions $f$ of unit $L^2_{x,y}$-norm and constant $\overline{C}$ such that
\begin{equation}\label{MaxPb}
\begin{split}
\Vert e^{it\partial_x \partial_y}f\Vert_{L^4_{x,y,t}}=\overline{C}:=\sup\{\Vert e^{it\partial_x \partial_y}g\Vert_{L^4_{x,y,t}}:\,\, \Vert g\Vert_{L^2_{x,y}}=1\}.\\
\end{split}
\end{equation}
The original version of this paper stated erroneously that Gaussians were optimizers for the Strichartz norm above, and gave a corresponding numerical value for $\overline{C}$. In fact, it was later proved in \cite{COS} that Gaussians are not critical points of the Strichartz norm and thus cannot be optimizers. One can however use the profile decomposition for Theorem \ref{ProfileDecTheorem} to prove existence of an extremizer (see \cite{Sh} for a similar proof).

\begin{proposition}
There exists $f\in L^2_{x,y}(\mathbf{R}^2)$ such that
\begin{equation}
\label{Maximizers}
\Vert e^{it\partial_x \partial_y}f\Vert_{L^4_{x,y,t}} = \overline{C} \Vert f\Vert_{L^2_{x,y}}.
\end{equation}
where $\overline{C}$ is given in \eqref{MaxPb}.
\end{proposition}

The exact nature of the extremizer $f$ above remains mysterious.
Preliminary numerical investigations confirm that the optimizer should be a genuine function of $x$ and $y$ and suggests that it has nice decay and smoothness properties.

We reproduce some plots of the amplitude of an extremizer and its Fourier transform below. These are obtained by maximizing the Strichartz norm among functions in the span of the first $25$ Hermite functions with unit $L^2$ norm. Some care has been given to obtain accurate computations and minimize boundary effects, but we do not mean that the result is anything but suggestive. However, it seems to indicate that extremizers are indeed smooth and localized and that they can be chosen {\it real valued}, which is why we have not reproduced their phase.

\begin{figure}[H]
\begin{subfigure}{.5\textwidth}
  \centering
  \includegraphics[width=\linewidth]{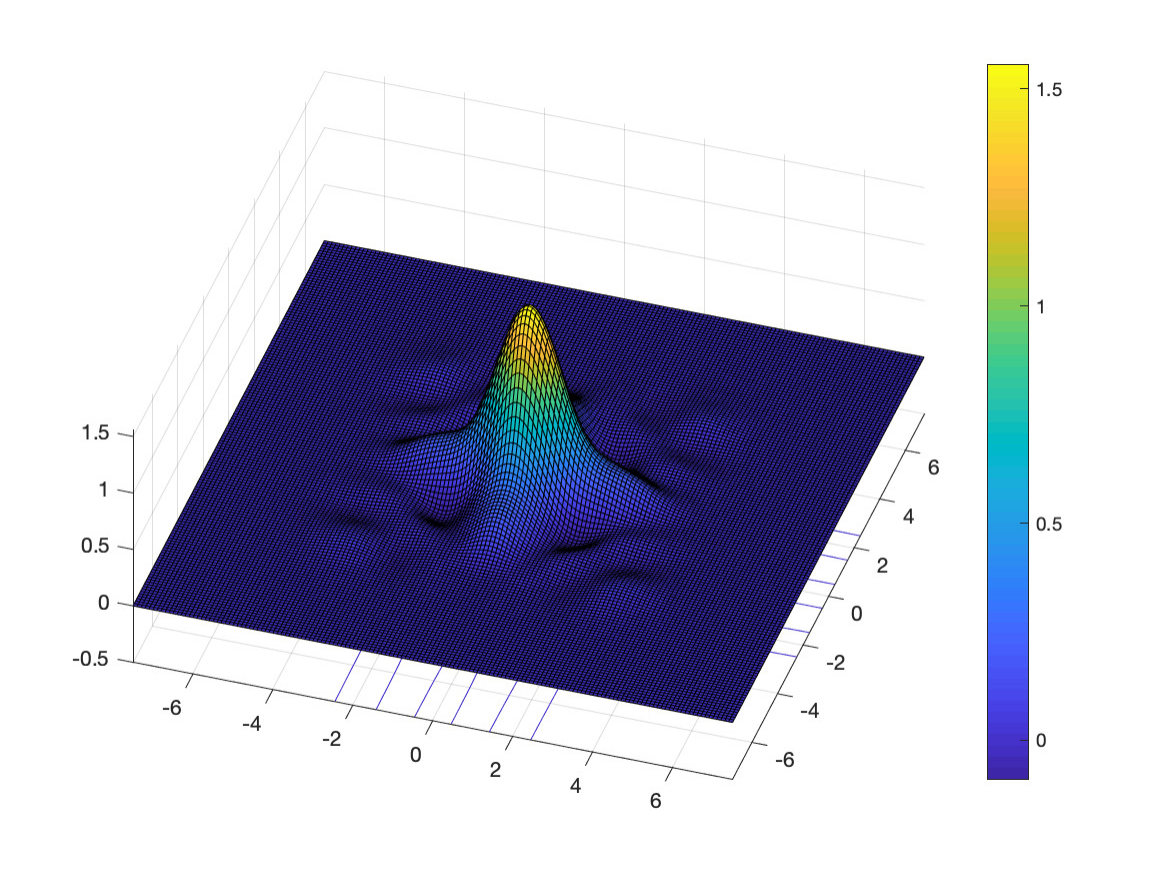}
  \caption{Plot of $\vert f\vert$}
\end{subfigure}%
\begin{subfigure}{.5\textwidth}
  \centering
  \includegraphics[width=\linewidth]{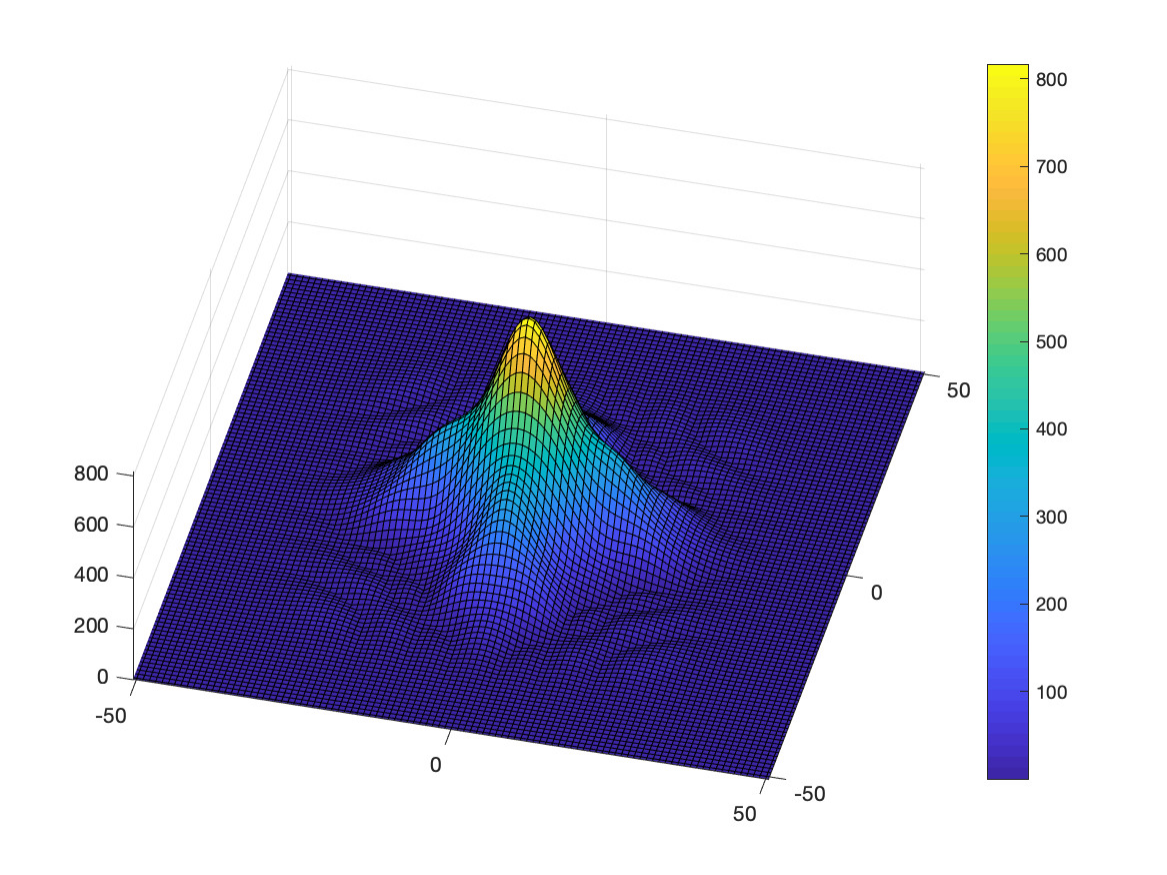}
  \caption{Plot of $\vert\widehat{f}\vert$}
\end{subfigure}
\\
\begin{subfigure}{.5\textwidth}
  \centering
  \includegraphics[width=\linewidth]{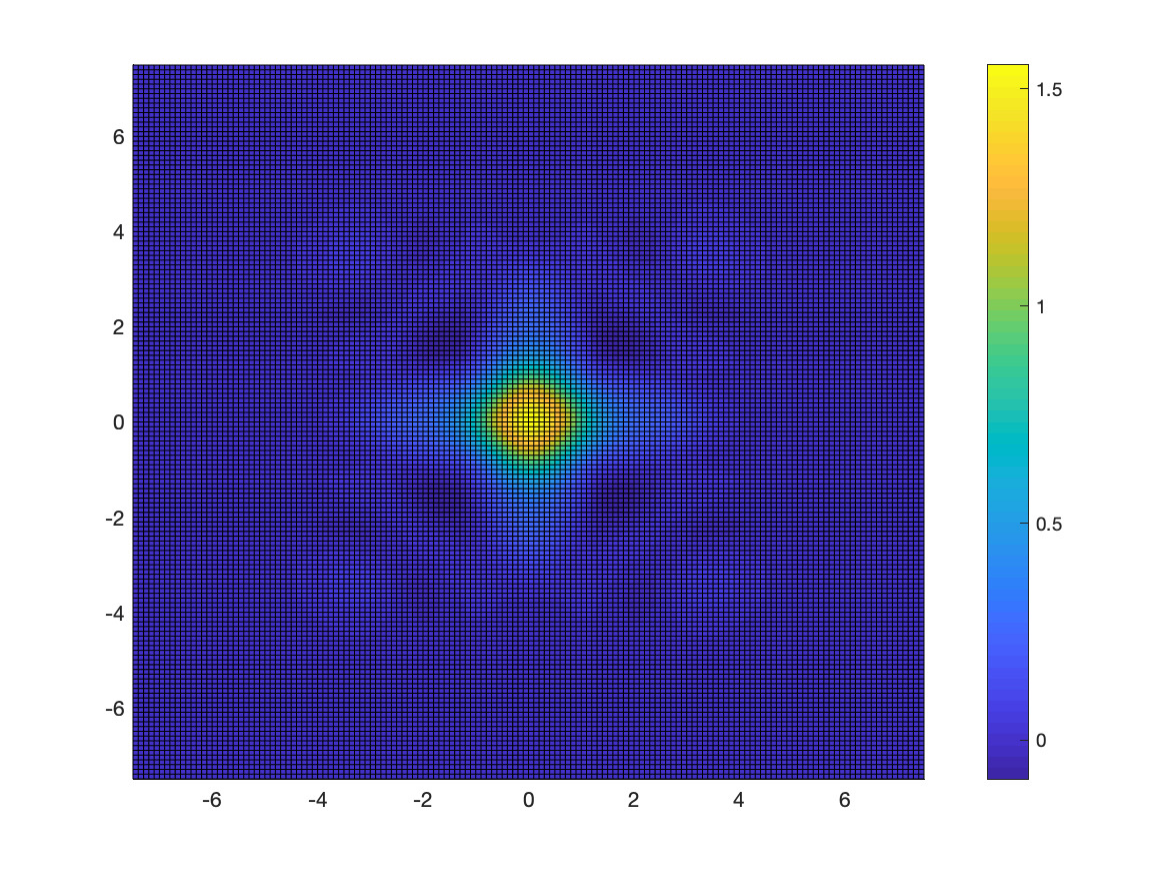}
  \caption{Heat map of $\vert f\vert$}
\end{subfigure}%
\begin{subfigure}{.5\textwidth}
  \centering
  \includegraphics[width=\linewidth]{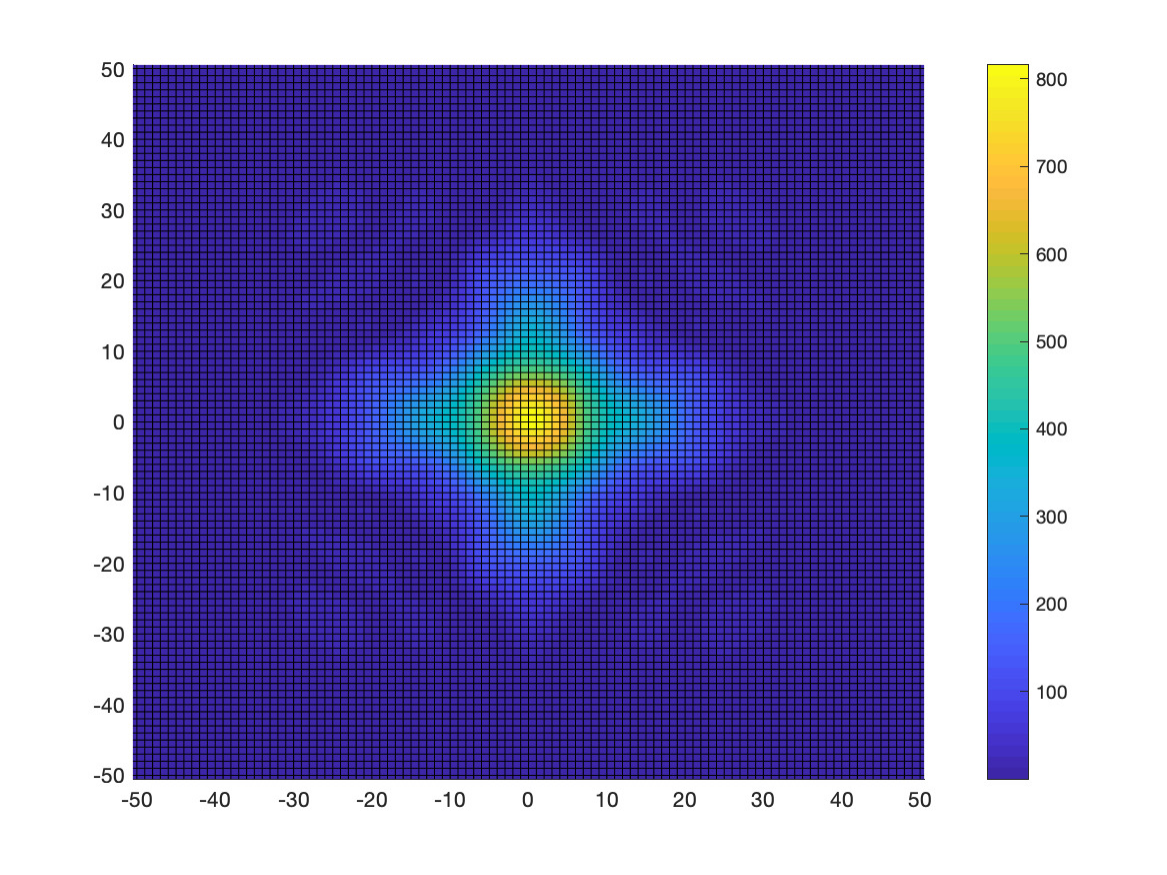}
  \caption{Heat map of $\vert\widehat{f}\vert$}
\end{subfigure}
\caption{Various views of an optimizer $f$ for \eqref{MaxPb}.}
\label{fig:fig}
\end{figure}


\begin{thebibliography}{MA}




\bibitem{BG} {\sc H. Bahouri, and P. G\'erard}, {\em High frequency approximation of solutions to critical
nonlinear wave equations}, Amer. J. Math. {\bf 121} (1999), 131-175.

\bibitem{Bou1} {\sc J. Bourgain}.
{\em Refinements of Strichartz inequality and applications to 2D-NLS with critical nonlinearity},
International Mathematics Research Notices, {\bf 8} (1998), 253-283.

\bibitem{COS} {\sc E. Carneiro, L. Oliviera and M. Sousa}. {\em Gaussians never extremize Strichartz inequalities for hyperbolic paraboloids}, preprint (2019).

\bibitem{Cazenave} {\sc T. Cazenave}. {\em Semilinear Schr\"odinger Equations}, Courant Lecture Notes in Mathematics, {\bf 10}, AMS Publishing, Providence (2003). 

\bibitem{Chihara} {\sc H. Chihara}. {\em Smoothing effects of dispersive pseudodifferential equations},
Comm. Partial Differential Equations, {\bf 27}, No. 9-10 (2002), 1953-2005. 

\bibitem{Dod1} {\sc B. Dodson},  {\em Global well-posedness and scattering for the defocusing, $L^{2}$-critical, nonlinear Schr\"odinger equation when $d = 2$}, Duke Math J., {\bf 165}, No. 10 (2016), 3435-3516.  

\bibitem{Dod2} {\sc B. Dodson},  {\em Global well-posedness and scattering for the mass critical nonlinear Schr\"odinger equation with mass below the mass of the ground state}, Adv. Math., {\bf 285} (2015), 1589-1618. 

\bibitem{FV} {\sc L. Fanelli and N. Visciglia}.  {\em The lack of compactness in the Sobolev-Strichartz inequalities}. J. Math. Pures et Appl, {\bf 99}, No. 3, (2013), 309-320.

\bibitem{Gal1} {\sc I. Gallagher}. {\em Profile decomposition for solutions of the Navier-Stokes equations}, Bull.
Soc. Math. France, {\bf 129} (2001), 285-316.

\bibitem{G1} {\sc P. G\'erard},  {\em Description du defaut de compacite de l'injection de Sobolev,} ESAIM Control Optim. Calc. Var. {\bf 3} (1998), 213-233.

\bibitem{GS} {\sc M. Ghidaglia and J.C. Saut},  {\em Nonexistence of Travelling Wave Solutions to Nonelliptic Nonlinear Schrodinger Equation}, J. Nonlinear Sci. {\bf 6} (1996), 139--145.

\bibitem{HuZh} {\sc D. Hundertmark and V. Zharnitsky},  {\em On sharp Strichartz inequalities in low dimensions}, International Mathematics Research Notices {\bf 2006} (2006), 34080.

\bibitem{Iftimie} {\sc Dragos Iftimie},  {\em A Uniqueness Result for the Navier--Stokes Equations with Vanishing Vertical Viscosity}, SIAM Journal on Mathematical Analysis {\bf 33}, No. 6 (2002),  1483--1493.

\bibitem{KM1} {\sc C.E. Kenig and F. Merle}, {\em Global well-posedness, scattering and blow-up for the energy-critical, focusing, non-linear Schr\"odinger equation in the radial case}, Inventiones Mathematicae {\bf 166}, No. 3 (2006), 645-675.

\bibitem{KPRV1} {\sc C. E. Kenig, G. Ponce, C. Rolvung, and L. Vega}. {\em The
    general quasilinear ultrahyperbolic Schr\"odinger equation}.
  Adv. Math. {\bf 196}, No. 2 (2005), 402-433.

\bibitem{Ker1} {\sc S. Keraani}, {\em On the Defect of Compactness for the Strichartz Estimates of the
Schr\"odinger Equations}, Journ. Diff. Eq. {\bf 175} (2001), 353-392.

\bibitem{KV} {\sc R. Killip and M. Visan},  {\em Nonlinear Schr\"odinger Equations at Critical Regularity}, {\em Clay Summer School Lecture Notes}, (2008). 

\bibitem{Lee}  {\sc S. Lee},  {\em Bilinear restriction estimates for surfaces with curvatures of different signs}, Transactions of the American Mathematical Society {\bf 358}, No. 8 (2006), 3511-3533.

\bibitem{LP} {\sc F. Linares and G. Ponce}, {\em On the Davey-Stewartson systems}, Annales de l'Institut Henri Poincare (C) Non Linear Analysis. {\bf 10}, No. 5 (1993), 523-548.

\bibitem{Lions1} {\sc P.-L. Lions}. {\em The concentration-compactness principle in the calculus of variations. The limit case, Part I},  Revista Matem\'atica Aberoamericana, {\bf 1}, Issue 1 (1985), 145-201.

\bibitem{Lions2} {\sc P.-L. Lions}. {\em The concentration-compactness principle in the calculus of variations. The limit case, Part II},  Revista Matem\'atica Aberoamericana, {\bf 1}, Issue 2 (1985), 45-121.

\bibitem{KNZ} {\sc P. Kevrekidis, A. Nahmod and C. Zeng}.  {\em Radial standing and self-similar waves for the hyperbolic cubic NLS in 2D}, Nonlinearity {\bf 24} (2011), No. 5, 1523-1538. 


\bibitem{MMT3} {\sc J.L. Marzuola, J. Metcalfe, D. Tataru}.  {\em Quasilinear Schr\"odinger equation I:  small data and quadratic interactions}.  Adv. Math., {\bf 231}, No. 2 (2012), 1151-1172. 

\bibitem{MerleVega} {\sc F. Merle and L. Vega}.  {\em Compactness at blow - up time for $L^{2}$ solutions of the critical nonlinear Schr{\"o}dinger equation in 2D}.  International Mathematics Research Notices, No. 8 (1998), 399-425.

\bibitem{MVV} {\sc A. Moyua, A. Vargas and L. Vega}, {\em Schr\"odinger maximal function and restriction properties of the Fourier transform}, International Mathematics Research Notices, 1996, 793--815.



\bibitem{RV} {\sc K.M. Rogers and A. Vargas}.  {\em A refinement of the Strichartz inequality on the saddle and applications}.  J. Functional Anal., {\bf 241}, No. 2 (2006), 212-231. 
 
\bibitem{RS}  {\sc M. Ruzhansky and M. Sugimoto}.  {\em Smoothing properties of evolution equations via canonical transforms and comparison principle}, Proceedings of the London Mathematical Society {\bf 105}, No. 2 (2012), 393-423. 

\bibitem{S-T} {\sc Monique Sabl\'e-Tougeron}, {\em R\'egularit\'e microlocale pour des problemes aux limites non lin\'eaires}, Ann. Inst. Fourier {\bf 36}, No. 1 (1986), 39--82.

\bibitem{Sh} S. Shao, {\em Maximizers for the Strichartz and the Sobolev-Strichartz inequalities for the Schr\"odinger equation}, Electronic Journal of Differential Equations (EJDE), (2009)
Volume: 2009, page Paper No. 03, 13 p.

\bibitem{SS} {\sc C. Sulem and P. Sulem}.  {\em Nonlinear Schr\"odinger Equations}, Springer (1999).

\bibitem{Tao1} {\sc Terence Tao}. {\em A sharp bilinear restriction estimate for paraboloids}. Geometric \& Functional Analysis, {\bf 13}, No. 6 (2003), 1359-1384.

\bibitem{T2} {\sc N. Totz}.  {\em A justification of the modulation approximation to the 3D full water wave problem}.  Communications in Mathematical Physics, {\bf 335}, No. 1 (2015), 369-443.

\bibitem{T1} {\sc N. Totz}.  {\em Global Well-Posedness of 2D Non-Focusing Schrodinger Equations via Rigorous Modulation Approximation}.  J. Diff. Eq., {\bf 261}, No. 4 (2016), 2251-2299.

\bibitem{TW} {\sc N. Totz and S. Wu}.  {\em A rigorous justification of the modulation approximation to the 2D full water wave problem}.  Comm. Math. Phys., {\bf 310}, No. 3 (2012), 817-883.
    
\bibitem{V1} {\sc A. Vargas}.  {\em Restriction theorems for a surface with negative
curvature}.  Math. Z., {\bf 249}, (2005), 97-111.

\end{thebibliography}
\end{document}